\documentclass[11pt, reqno]{amsart}
\usepackage[all,tips]{xy}
\usepackage{latexsym, amsfonts, amsmath, amssymb, mathrsfs, graphicx, xcolor, ytableau, hyperref, float}
\usepackage[justification=centering]{caption}
\allowdisplaybreaks
\usepackage{centernot}
\usepackage{tikz}

\definecolor{red}{rgb}{1,0,0}
\definecolor{magenta}{rgb}{1,0,1}
\definecolor{dartmouthgreen}{rgb}{0.05, 0.5, 0.06}
\definecolor{purple(x11)}{rgb}{0.63,0.36,0.94}
\definecolor{turquoise}{rgb}{0.25, 0.87, 0.81}
\newtheorem{theorem}{Theorem}[section]
\newtheorem{lemma}[theorem]{Lemma}
\newtheorem{proposition}[theorem]{Proposition}

\newtheorem{conjecture}[theorem]{Conjecture}
\newtheorem{definition}[theorem]{Definition}
\newtheorem*{problem}{Open Problem}
\newtheorem*{principle}{Principle}

\theoremstyle{definition}
\newtheorem{remark}[theorem]{Remark}

\newcommand{\Cal}[1]{\ensuremath{\mathcal{#1}}}

\newcommand{\lp}{\left(}
\newcommand{\rp}{\right)}

\usepackage[textheight=8.75in, textwidth=6.75in]{geometry}

\def\C{{\mathbb C}}

\def\N{{\mathbb N}}
\def\Z{{\mathbb Z}}
\def\Q{{\mathbb Q}}
\def\F{{\mathbb F}}

\def\O_K{{\Cal{O}_{K}}}
\def\O_F{{\Cal{O}_{F}}}
\def\N_F{{\Cal{N}_{F/\Q}}}

\newcommand{\cE}{{\mathcal E}}
\newcommand{\ord}{\mathrm{ord}}

\def\O_K{{\Cal{O}_{K}}}
\def\O_F{{\Cal{O}_{F}}}
\def\N_F{{\Cal{N}_{F/\Q}}}

\newcommand{\SL}{\mathrm{SL}}

\DeclareMathOperator{\li}{\mathrm{li}}

\numberwithin{equation}{section}
\numberwithin{theorem}{section}

\title{New Types of Sturm bounds via $p$-adic transfer methods}

\author{William Craig}
\address{Department of Mathematics, United States Naval Academy, 572C Holloway Road
Mail Stop 9E. Annapolis, MD 21402}
\email{wcraig@usna.edu}


\begin{document} 
	
\begin{abstract} 
    Sturm's theorem states that a modular form with coefficients in $\Z$ or $\Z/m\Z$ can only have an explicitly bounded order of vanishing at infinity. This result is one of the most powerful computational tools in the study of modular forms, and has widespread applications to congruences and other kinds of explicit calculations in mathematics and physics. In this paper, we formulate a new ``$p$-adic transfer method" that lifts Sturm-type bounds from one space to another using exclusively non-geometric inputs. As an application, we transfer the Sturm bounds for classical modular forms to the space of quasimodular forms of level one. These bounds are applicable uniformly for quasimodular forms with coefficients in $\Z$ or $\Z/m\Z$, which extends the non-uniform results for $\Z/m\Z$ only which can be derived from classical theories. We also discuss the potential for future applications to other quasi- and mixed-weight modular objects, and perhaps even entirely non-modular objects.
\end{abstract}
 
\maketitle

\section{Introduction and Statement of results}

Modular forms and their Fourier series are among the most fascinating and enduring objects in mathematics, especially in number theory, in which they play crucial roles in many of the most important problems of the last two centuries. Modular forms arise in many contexts and have many equivalent constructions, often given in terms of special functional equations, functions on elliptic curves, or as sections of line bundles. For our work, we will define a modular form by its Fourier expansion; that is, a {\it modular form} is an element of the algebra $\C[E_4, E_6]$, where $E_{2k}$ for $k \geq 2$ are the {\it Eisenstein series}
\begin{align*}
    E_{2k}(z) = 1 - \dfrac{4k}{B_{2k}} \sum_{n \geq 1} \sigma_{2k-1}(n) q^n,
\end{align*}
where $\sigma_\ell(n) := \sum_{d|n} d^\ell$ are the sum-of-divisors functions, $B_{2k}$ are the even index Bernoulli numbers, and where $q := e^{2\pi i z}$ is a complex variable with $z$ in the upper half plane. In the larger theory of modular forms, these are usually called modular forms for the full modular group $\mathrm{SL}_2(\Z)$, or simply modular forms of level one. The algebra of modular forms is graded by a {\it weight}, defined to be $2k$ for each $E_{2k}$ and extended so that any product $E_4^a E_6^b$ carries weight $4a+6b$. There are many generalizations of modular forms which can be defined, the most immediate of which are the modular forms for congruence subgroups $\Gamma \subset \mathrm{SL}_2(\Z)$, which are constructed from variants of Eisenstein series \cite[Ch. 4]{DiamondShurman}.

All modular forms for $\SL_2(\Z)$ have Fourier expansions
\begin{align*}
    f = \sum_{n \geq 0} a_f(n) q^n,
\end{align*}
which may be expressed, for instance, as linear combinations of convolution sums of divisor sums. Modular forms and their generalizations are often generating functions for interesting objects in arithmetic geometry, partition theory, and many other areas of mathematics. We refer the reader to texts such as \cite{CohenStromberg,DiamondShurman,OnoWeb} for more on the general theory of modular forms; we will introduce more details on modular form theory as they are required.

In the arithmetic and computational theories of modular forms, one of the most important tools is the {\it Sturm bound}. Originally written down by Sturm \cite{Sturm}, this bound reduces the identification of modular forms by their Fourier series to an explicit finite computation. In its most common form, this bound is a tool for studying congruences between modular forms \cite{AhlgrenBoylan,Calegari,Murty}, both in specific instances and in parameterized families. The Sturm bound has been applied fruitfully throughout numerous areas of mathematics, including in the study of arithmetic properties of partitions \cite{AhlgrenBoylan,OnoRamanujan,OnoParity,Radu}, mock modular forms \cite{DuncanGriffinOno}, modular representation theory \cite{GranvilleOno}, Ap\'{e}y series \cite{AhlgrenOno}, mathematical physics \cite{CollierEberhardtMuhlmannRodriguez,HarveyWu}, arithmetic geometry \cite{CremonaMazur,RouseThorner}, class numbers \cite{Beckwith,OnoClassReal}, moonshine \cite{DuncanGriffinOno}, Iwasawa theory \cite{Sharifi}, $L$-functions \cite{BeneishDarmonGehrmannRoset}, and $q$-multiple zeta values \cite{AmdeberhanAndrewsTauraso,AmdeberhanOnoSingh}; applications go on almost without limit.

The Sturm bound is usually stated in terms of the {\it order of vanishing} of a Fourier series, which we define for $f = \sum_{n \geq 0} a_f(n) q^n$ by
\begin{align*}
    \ord_\infty(f) = \min\{ n \in \Z : a_f(n) \not = 0 \}.
\end{align*}
We use the same notation regardless of which ring contains the coefficients of $f$. Then the classical Sturm bound is stated as follows.

\begin{theorem}[{\cite[Theorem 2.58]{OnoWeb}}] \label{T: Classical Sturm}
    Let $k \geq 1$ be an integer and let $\Gamma$ be a congruence subgroup of $\mathrm{SL}_2(\Z)$ with finite index $[\mathrm{SL}_2(\Z):\Gamma]$. Let $f$ be a holomorphic modular form for $\Gamma$ of weight $k$ with integer coefficients. The following are true:
    \begin{enumerate}
        \item $f$ is determined uniquely by its first $\frac{[\mathrm{SL}_2(\Z):\Gamma] k}{12}+1$ Fourier coefficients. Equivalently, if $\ord_\infty(f) > \frac{[\mathrm{SL}_2(\Z):\Gamma] k}{12}$, then $f = 0$.
        \item For any integer $m \geq 1$, $f \pmod{m}$ is determined uniquely by its first $\frac{[\mathrm{SL}_2(\Z):\Gamma] k}{12}+1$ Fourier coefficients modulo $m$. Equivalently, if $\ord_\infty(f) > \frac{[\mathrm{SL}_2(\Z):\Gamma] k}{12}$, then $f \equiv 0 \pmod{m}$.
    \end{enumerate}
\end{theorem}

This theorem has two classical proofs. One proof is linear algebraic, constructing an explicit $\Z$-basis for $M_{2k}$ with elements $f_i$ having Fourier expansions $q^i + O(q^{i+1})$. The other well-known approach is by the {\it valence formula} \cite[Theorem 1.29]{OnoWeb}, which says that a modular form can only have an explicitly bounded number of zeros in its fundamental domain. The second proof in particular can be cast in geometric terms in more general settings (i.e. Riemann--Roch), which has enabled various generalizations of the Sturm bound to be studied for modular objects arising from geometry. Such results exist for Hilbert modular forms \cite{Takai}, Siegel modular forms \cite{BanerjeeAironVishwakarmaDebnath,RichterRaum}, and modular forms over function fields \cite{ArmanaWei}. These generalizations find many applications in themselves, including to the identification of isogenous elliptic curves \cite{ArmanaWei}.

It is noteworthy to point out that while many of these problems address only individual congruences or individual calculations, some deal with congruences modulo all primes at once, famously including Ahlgren and Boylan's classification of Ramanujan-type congruences for $p(n)$ \cite{AhlgrenBoylan}, which implicitly uses Sturm-type results for families of modular forms of weight $\frac{\ell^2-1}{2}+4$ in order to study the congruence structure of $p(\ell n + \beta)$ modulo $\ell$. Thus, questions not simply of finite computability, but also of uniformity of the bound with respect to the modulus, are central both conceptually and practically. One can also see Sturm's theorem for more general congruence subgroups appear in related work on classification of congruences for generalized Frobenius partitions \cite{AhlgrenCastillo}.

The objective of this work is to formulate a new approach to deriving Sturm-type bounds, which does not rely on underlying geometry or holomorphic transformations. As such, this method will have applicability in new domains. The basic structure of our argument relies instead on the following input:
\begin{enumerate}
    \item A graded family of $q$-series $\mathcal F$ with finite-dimensional fixed-weight components for which we wish to obtain a Sturm bound.
    \item Suitably robust arithmetic relationships modulo an infinite family of primes (or prime powers) between $\mathcal F$ and a some other graded space $\mathcal G$ already possessing a Sturm-type bound modulo those primes (or prime powers).
    \item Bounds on Fourier coefficients of some particular basis for the graded subspaces of $\mathcal F$.
\end{enumerate}
For such families $\mathcal F$ and $\mathcal G$ of $q$-series, we describe a ``$p$-adic transfer" procedure whereby the Sturm-type bound for $\mathcal G$ is transferred to a Sturm-type bound for $\mathcal F$. The strength and flexibility of the new bound depends on structural properties of the basis used for graded subspaces of $\mathcal F$, which gives our method a wide scope for application in situations where computations are difficult. For instance, a known $\Z$-basis permits the derivation of Sturm bounds in $\Z$ and $\Z/m\Z$ automatically, whereas weaker basis properties (like a basis belonging to $\Z[[q]]$ which is not a $\Z$-basis) can still yield Sturm bounds for $\Z$. Neither does one not need any explicit formulas for the coefficients of the basis of $\mathcal F$ being used; it is enough to rely on upper bounds for said coefficients, which can be easier to obtain at times. These hypotheses are different from (and often weaker than) what Sturm's theorem or any of its analogs usually require, and therefore should be applicable to a much wider variety of spaces. In fact, we hope this method will apply even to spaces of $q$-series not directly related to modularity.

We demonstrate our method by working in the algebra of quasimodular forms (important in number theory, enumerative geometry and mathematical physics) using two distinct basis constructions to demonstrate the relative strengths of different inputs to the algorithm. distinctive basis properties to demonstrate the flexibility of the method. For a congruence subgroup $\Gamma$, a {\it quasimodular form} is a polynomial in the Eisenstein series $E_2$ with coefficients which are modular forms for $\Gamma$. Like modular forms, quasimodular forms possess a weight; we extend the definition given to modular forms by giving $E_2$ the weight 2. We will then use the notation $\widetilde{M}_{2k}$ to denote the vector space of quasimodular forms of weight $2k$, and we use $\widetilde{M}_{2k}^p$ and $\widetilde{M}_{2k}^\Z$ to denote that coefficients should be restricted to integers modulo $p$ or simply to the integers, respectively. Quasimodular were introduced by Kaneko and Zagier \cite{KanekoZagier} and were proven to count certain maps between curves. In the proceeding decades, quasimodular forms have made numerous appearances, to fields such as topology \cite{BlochOkounkov,OkounkovPandharipande,Zagier}, multiple zeta values and their $q$-analogs \cite{BachmannKuhn}, partitions  \cite{BringmannCraigIttersumPandey}, mock modular forms \cite{ImamogluRaumRichter}, prime numbers \cite{CraigIttersumOno}, mathematical physics \cite{AdamsWeinzierl,Dijkgraaf,DijkgraafHollandsSulkowskiVafa,KlemmManschot,DijkgraafBabyUniverses}, elliptic curves \cite{Dijkgraaf}, deformation theory \cite{Mazur}, and transcendence theory \cite{MurthRoth}.

Quasimodular forms share many similarities with modular forms, as for instance the coefficients of $E_2$ are also built from divisor sums, but many differences emerge. In particular, $E_2$ does not possess a holomorphic transformation law, as a modular form does, and therefore the geometric approaches to Sturm bounds do not proceed in like fashion. The only Sturm bounds in the literature of which the author is aware are due to van Ittersum and Ringeling \cite{IttersumRingeling}, who derive such a bound for quasimodular forms of the form $f + g E_2$, where $f \in M_{2k}$ and $g \in M_{2k-2}$. This covers only a small fraction of all quasimodular forms, and their proof, which generalizes the valence proof of the classical version, would be difficult to generalize to permit arbitrary powers of $E_2$.

It is the goal of our work to prove the first completely general theorem of Sturm type for quasimodular forms. We may state explicitly our first Sturm bound as follows.

\begin{theorem} \label{T: Main}
   Let $k \geq 1$ be an integer, and let $f$ be a quasimodular form for $\SL_2(\Z)$ weight $2k$. Then the following are true:
    \begin{enumerate}
        \item $f$ is determined uniquely by its first $\left\lceil \frac{k^5 \log(k)}{12} \right\rceil + 1$ Fourier coefficients. Equivalently, if $\ord_\infty(f) > \left\lceil \frac{k^5 \log(k)}{12} \right\rceil$, then $f = 0$.
        \item For any integer $m \geq 1$, $f \pmod{m}$ is determined uniquely by its first $\left\lceil \frac{k^5 \log(k)}{12} \right\rceil + 1$ Fourier coefficients modulo $m$. Equivalently, if $\ord_\infty(f \bmod{m}) > \left\lceil \frac{k^5 \log(k)}{12} \right\rceil$, then $f \equiv 0 \pmod{m}$.
        \end{enumerate}
\end{theorem}

\begin{remark}
    We make the following remarks:
    \begin{enumerate}
        \item As noted in Lemma \ref{L: Quasimodular to modular}, a quasimodular form of weight $2k$ is congruent modulo $p$ to a modular form of weight $2kp$, and Theorem \ref{T: Classical Sturm} can thus improve the bound for small primes. However, Theorem \ref{T: Main} has the important benefits that it is applicable independently of the modulus, and that it also works as a bound over $\Z[[q]]$. Furthermore, future improvements to the method (especially in certain determinant calculations) have the potential to achieve bounds much closer to the optimal one, as can be seen in Theorem \ref{T: Sturm Q Only}.
        \item We note that, in this work, we do not state nor prove any Sturm bound for congruence subgroups of $\mathrm{SL}_2(\Z)$. Strictly speaking, the fact that the dimension of spaces of quasimodular forms grows quadratically in weight aspect causes the classical approach of Sturm \cite[pg. 276--277]{Sturm} to break down, since taking a ``trace-form" only introduces a linear factor of the index of $\Gamma$ in $\SL_2(\Z)$, whereas a higher degree factor would be required. Congruence subgroups can, however, be managed using precisely our method; this is left for future work.
    \end{enumerate}
\end{remark}

Our method of proof does not rely on generalizing either of the two classical approaches towards the derivation of Sturm bounds for modular forms. The method is completely different, and builds on interplay between bounds in matrix analysis, growth rates of Fourier coefficients of quasimodular forms, and the arithmetic theory of modular forms modulo $p$ due to Serre \cite{Serre} and Swinnerton-Dyer \cite{SwinnertonDyer}. Using these tools, we create a kind of recursive descent process capable of establishing {\it ex nihilo} finite (but excessively large) initial bounds, and subsequently to refine this very large bound to a much smaller one. The end result of this technique is a kind of equilibrium point of this algorithm. The size of this stable point is related to the size of the Fourier coefficients and the relative knowledge we possess about the linear-algebraic relations (mainly determinants) of some particular basis.

Theorem \ref{T: Main} is in strict analogy with the classical theorem of Sturm: namely, the bound on the order of vanishing at infinity for $f$ does not depend on whether one investigates $f$ as a series in $\Z[[q]]$ or modulo $m$ for some $m \geq 2$. The most elementary explanation for this classical fact is that one can produce a $\Z$-basis for $M_{2k} \cap \Z[[q]]$ whose Fourier expansions have integer coefficients and take the form $1 + O(q), q + O(q^2), q^2 + O(q^3)$, and so on. Such a useful basis has never been constructed explicitly for quasimodular forms, and the method for constructing such a basis for modular forms does not generalize.\footnote{For modular forms, this construction relies properties of Ramanujan's $\Delta$-function; no function playing a similar role exists for quasimodular form theory. More on this in Section \ref{Sec: Nuts and Bolts}.} The method derived in this paper, however, permits some degree of flexibility, such that if we permit ourselves to ignore the modulo $m$ aspect that is traditionally associated with Sturm bounds, we can improve the bound in Theorem \ref{T: Main} in the case when coefficients are in $\Z$ rather than $\Z/m\Z$ (even in such cases, this may still be sufficient for proving congruence theorems).

\begin{theorem} \label{T: Sturm Q Only}
    Let $k \geq 1$ be an integer, and let $f$ be a quasimodular form for $\SL_2(\Z)$ weight $k$. Then $f$ is determined uniquely by its first $\left\lceil \frac{k^4 \log(k)}{2} \right\rceil +1$ Fourier coefficients. Equivalently, if $\ord_\infty(f) > \left\lceil \frac{k^4 \log(k)}{2} \right\rceil$, then $f = 0$.
\end{theorem}

The proof of Theorem \ref{T: Sturm Q Only} follows on nearly identical lines to that of Theorem \ref{T: Main}. Namely, both proofs produce Sturm bounds out of the same algorithmic procedure applied to a different $\Q$-basis for spaces of quasimodular forms; Theorem \ref{T: Main} then additionally extends to congruences because this $\Q$-basis is also a $\Z$-basis, while the basis used to prove Theorem \ref{T: Sturm Q Only} is only a $\Q$-basis. The differences in the resulting bounds arise from the simple fact that the $\Z$-basis used is only constructed algorithmically, and not much of use can be said about it. More explicit constructions for a $\Z$-bases may yield improved results.

Neither Theorem \ref{T: Main} nor Theorem \ref{T: Sturm Q Only} should be taken as optimal (see Section \ref{Sec: Conclusion}), and there are indeed some cases in which Theorem \ref{T: Classical Sturm} alongside the fact that $E_2 \equiv E_{p+1} \pmod{p}$ for any prime $p \geq 3$ (see Lemma \ref{L: Quasimodular to modular}) sometimes can provide more precise information if specific knowledge of a case at hand is known. Theorem \ref{T: Main} (2), however, does not depend on the modulus under consideration and may therefore be used in scenarios where large moduli must be considered. Beyond this, no tools prior to our results have given any kind of Sturm bound for quasimodular forms in with coefficients in $\Z$. Since many fields in which quasimodular forms arise are concerned with counting problems \cite{KanekoZagier} and not with arithmetic partition-style congruences, this distinction is critical to understanding the strengths of this technique. With Theorems \ref{T: Main} and \ref{T: Sturm Q Only} in hand, one may now identify not merely congruences for small primes using a finite number of their Fourier coefficients, but congruences for any primes and even exact formulas for quasimodular forms. The method is also flexible, such that the future development of more useful basis expansions will lead to improvements in the bounds of Theorems \ref{T: Main} and \ref{T: Sturm Q Only}.

In short, these tools provide a robust and rigorous framework to develop computational theories of quasimodular forms in particular, and prospectively many other theories of $q$-series, and provide uniformity which can be used in number-theoretic scenarios in which the modulus grows (e.g. study of families of congruences modulo $\ell^n$ or congruence classification problems of the Ahlgren-Boylan type).

The remainder of this paper proceeds as follows. In Section \ref{Sec: Nuts and Bolts}, we summarize important content from the existing literature which we require. In Section \ref{Sec: Bounds}, we give uniform upper bounds on Fourier coefficients on a basis of $\widetilde{M}_{2k}^\Z$. In Section \ref{Sec: Linear Algebra}, we utilize tools from linear algebra to establish criteria for linear independence of integer vectors. Section \ref{Sec: Proof} combines these tools into a $p$-adic transfer procedure that produces the bound in Theorem \ref{T: Main} as a ``stable point" to a descent argument. In Section \ref{S: Proof Q}, we modify some details of this proof to derive Theorem \ref{T: Sturm Q Only}. Finally, in Section \ref{Sec: Conclusion}, we discuss the theoretical boundaries of this method, possible extensions it may have to other areas both inside and outside modular form theories, and we provide numerical data, algorithms and conjectures concerning on optimal Sturm bounds for $\widetilde{M}_{2k}$.

\section*{Acknowledgments}

The author thanks Jan-Willem van Ittersum and Lukas Mauth for helpful discussions which have improved this manuscript. The views expressed in this article are those of the author and do not reflect the official policy or position of the U.S. Naval Academy, Department of the Navy, the Department of War, or the U.S. Government.

\section{Nuts and Bolts} \label{Sec: Nuts and Bolts}

\subsection{Modular forms basics}

We begin with a classical definition of various terms and vector spaces important in the discussions of modular and quasimodular forms. We note here that all modular forms for $\mathrm{SL}_2(\Z)$ have even weight.

\begin{definition}
    Let $k \geq 1$ be an integer.
    \begin{enumerate}
        \item A {\it modular form} of weight $2k$ for $\mathrm{SL}_2(\Z)$ is an analytic function $f(z)$ on the upper half plane, having a Fourier expansion in $q = e^{2\pi i z}$ with only non-negative powers of $q$, and satisfying the transformation law
        \begin{align*}
            f\lp \dfrac{az + b}{cz + d} \rp = \lp cz + d \rp^k f(z).
        \end{align*}
        The space of modular forms of weight $2k$ is denoted $M_{2k}$ and the space of cusp forms (those modular forms whose Fourier expansion has no $q^0$ term) of weight $2k$ is denoted $S_{2k}$. We also let $M_{2k}^\Z$ and $S_{2k}^\Z$ denote the spaces of forms having integer coefficients, e.g. $M_{2k}^\Z = M_{2k} \cap \Z[[q]]$.
    
        \item A {\it quasimodular form} of weight $2k$ is any analytic function $f(z)$ of the form
        \begin{align*}
            f = \sum_{j=0}^k E_2^j f_j,
        \end{align*}
        where $f_j$ is a modular form of weight $2k - 2j$. The space of quasimodular forms of weight $2k$ is $\widetilde{M}_{2k}$, and we likewise define $\widetilde{M}_{2k}^\Z = \widetilde{M}_{2k} \cap \Z[[q]]$.
    \end{enumerate}
    All modular forms and quasimodular forms have Fourier expansions $f(z) = \sum_{n \geq 0} a_f(n) q^n$, where $q = e^{2\pi i z}$.
\end{definition}

\begin{remark}
    In some works, a quasimodular form of weight $2k$ is taken to be any element of $\bigoplus\limits_{m \leq k} \widetilde{M}_{2m}$. The case of $\widetilde{M}_{2k}$ is then called the case of {\it pure weight}, and the more general space the case of {\it mixed weight}. In this paper, we always refer to quasimodular forms of pure weight unless specifically mentioned otherwise.
\end{remark}

In this section, we set notation and cover several simple facts about spaces of modular and quasimodular forms.

We recall two foundational examples in the theory of modular forms; the Eisenstein series and the Ramanujan $\Delta$-function.

\begin{definition}
    Let $k \geq 2$ be an even integer. We define the {\it Eisenstein series} of weight $k$ by
    \begin{align*}
        E_k(z) := 1 - \dfrac{2k}{B_k} \sum_{n \geq 1} \sigma_{k-1}(n) q^n, \ \ \ \ q = e^{2\pi i z},
    \end{align*}
    where $B_k$ is the $k$th Bernoulli number and $\sigma_{k-1}(n) = \sum\limits_{d|n} d^{k-1}$. We utilize in particular
    \begin{align*}
        E_2(z) &= 1 - 24 \sum_{n \geq 1} \sigma_1(n) q^n, \\
        E_4(z) &= 1 + 240 \sum_{n \geq 1} \sigma_3(n) q^n, \\
        E_6(z) &= 1 - 504 \sum_{n \geq 1} \sigma_5(n) q^n.
    \end{align*}
\end{definition}

\begin{definition}
    The Ramanujan $\Delta$-function is the modular form of weight 12 defined by
    \begin{align*}
        \Delta(z) := q \prod_{n \geq 1} \lp 1 - q^n \rp^{24} = \dfrac{E_4(z)^3 - E_6(z)^2}{1728} \in \Z[[q]].
    \end{align*}
    We denote the Fourier coefficients of $\Delta$ by the Ramanujan $\tau$-function:
    \begin{align*}
        \Delta(z) = \sum_{n \geq 1} \tau(n) q^n.
    \end{align*}
\end{definition}

The $\Delta$-function is the foundational example of a {\it cusp form}, meaning we have $\tau(0) = 0$; the space of all cusp forms will be denoted by $S_{2k}$. Because of the product formula for $\Delta$, is it known that $\Delta \in S_{12}^\Z$ and that $\Delta(z) \not = 0$ for any $z$ in the upper half plane. This fact implies in a straightforward manner that the map $f \mapsto \Delta f$ is an isomorphism $M_{2k} \to S_{2k+12}$. This fact, along with known computations for small weights, produces the following dimension formula:
\begin{align} \label{Eq: Valence}
    \dim\limits_\C M_k = \begin{cases}
        \left\lfloor \dfrac{k}{12} \right\rfloor + 1 & \text{if } k \not \equiv 2 \pmod{12}, \\[+0.5cm]
        \left\lfloor \dfrac{k}{12} \right\rfloor & \text{if } k \equiv 2 \pmod{12}.
    \end{cases}
\end{align}
It is then straightforward to demonstrate that $\Delta^a E_4^b E_6^c$ with $b,c$ chosen appropriately form a $\Z$-basis of $M_{2k}$, and indeed this is ``triangular basis" in the sense that the $q$-expansions of these basis elements are of form $q^a + O(q^{a+1})$.

We will rely on several methods of generating the space of modular forms; this particular basis will not be of much use to us. The most classical, which emerge precisely from the aforementioned isomorphism $f \mapsto f\Delta$, are the facts
\begin{align*}
    M_{2k} = \langle E_4^a E_6^b : 4a + 6b = 2k \rangle \ \ \ \text{and} \ \ \ \widetilde{M}_{2k} = \langle E_2^a E_4^b E_6^c : 2a + 4b + 6c = 2k \rangle.
\end{align*}
We note that $M_2 = \emptyset$, since $E_2$, the only Eisenstein series of weight 2, is not modular. We also note the classical {\it valence formula} for the dimension of $M_{2k}$ as a vector space. We can also characterize the dimension of the space of quasimodular forms using that of modular forms:
\begin{align} \label{Eq: Quasimodular Dimension}
    \dim\limits_\C \widetilde{M}_{2k} = \sum_{j=0}^k \dim_\C M_{2j} \leq \sum_{j=0}^k \lp \dfrac{j}{6}+1 \rp = \dfrac{k^2+k+12}{12}.
\end{align}

\subsection{Modular forms modulo $p$} \label{S: Modular forms modulo p}

A major component of our argument relies on the theory of modular forms modulo primes, particularly of odd primes. There are many beautiful developments of such theories from arithmetic and geometric perspectives by Serre \cite{Serre} and Katz \cite{Katz}, even extending to $p$-adic modular forms. We do not, however, need any of the stronger tools of the full $p$-adic theories; only the reduction modulo $p$, which is explained fully in Serre \cite{Serre}\footnote{It is very possible that using additional structure modulo higher powers of primes may be of use in this method; we rely only on results modulo $p$ for simplicity.}. In particular, key elements of these theories lie behind the resolution of Ramanujan's famous conjecture that the only congruences for $p(n)$ of the type $p\lp \ell n + \beta \rp \equiv 0 \pmod{\ell}$ occur for $\ell = 5, 7, 11$ and additional approaches to proving that such congruences exist; see in particular \cite{AhlgrenBoylan,Calegari,Radu}. We will lay out just those first elementary bits we need; namely, of the existence of congruences for Eisenstein series modulo primes.

\begin{lemma}[{\cite[Lemma 1.22]{OnoWeb}}] \label{L: Eisenstein congruences}
    Let $p \geq 3$ be prime. Then 
    \begin{align*}
        E_{p+1} \equiv E_2 \pmod{p} \ \ \ \text{and} \ \ \ E_{p-1} \equiv 1 \pmod{p}.
    \end{align*}
\end{lemma}

We utilize this elementary result to prove the following fact about quasimodular forms.

\begin{lemma} \label{L: Quasimodular to modular}
    Let $f \in \widetilde{M}_{2k}$ for $k \geq 1$, and let $p \geq 3$ be prime. Then there exists $g \in M_{2kp}$ such that $f \equiv g \pmod{p}$.
\end{lemma}

\begin{proof}
    By the definition of $\widetilde{M}_{2k}$, we have $f = \sum_{j=0}^k f_j E_2^j$ for some modular forms $f_j \in M_{2k-2j}$. We then have by Lemma \ref{L: Eisenstein congruences}, for each $j$, that
    \begin{align*}
        f_j E_2^j \equiv g_j := f_j E_{p+1}^j E_{p-1}^{2k-2j} \pmod{p},
    \end{align*}
    and that each $g_j$ has weight $(2k-2j) + j(p+1) + (2k-j)(p-1) = 2kp$. Thus, we let $g := \sum_{j=0}^k g_j \in M_{2kp}$ and it is apparent that $f \equiv g \pmod{p}$.
\end{proof}

In fact, something much stronger than Lemmas \ref{L: Eisenstein congruences}  and \ref{L: Quasimodular to modular} exists, due to Swinnerton-Dyer \cite{SwinnertonDyer}. The summary version of the result is that all congruences modulo $p$ between (quasi-)modular forms arise in some way from the congruence $E_{p-1} \equiv 1 \pmod{p}$ (i.e. from the ``Hasse invariant" \cite{Calegari}). This result is of substantial use in the theory of modular forms modulo $p$; we give now the required background to state this result. To state these results conveniently, we use the notations $\F_p$ when we are concerned with the integers modulo $p$

\begin{proposition} \label{P: Serre fact}
    Let $p \geq 5$ be a prime, $k \geq 0$ an integer, and let $M_{2k}$ and $M_{2k}^p$ denote the spaces of modular forms of weight $2k$ over $\C$ and $\F_p$, respectively. Let $M = \bigoplus\limits_{k \geq 0} M_{2k}$ and $M^p = \bigoplus\limits_{k \geq 0} M_{2k}^p$ denote the full algebras of all modular forms over $\C$ and $\F_p$, respectively. Then the following are true: 
    \begin{enumerate}
        \item There are injective maps $M^p_{2k} \hookrightarrow M^p_{2k+p-1}$ given by $f \mapsto f E_{p-1}$.
        \item Let $M^{(\alpha)}$ denote, for $\alpha \in \Z/(p-1)\Z$, the union of all modular forms of weight congruent to $\alpha \pmod{p-1}$. Then $M^p = \bigoplus\limits_{\alpha \in \Z/(p-1)\Z} M^{(\alpha)}$.
        \item Let $A(Q,R)$ be the isobaric weight $p-1$ polynomial such that $A(E_4, E_6) = E_{p-1}$, and let $\overline{A}$ denote the reduction of this polynomial modulo $p$. Then there is an isomorphism
        \begin{align*}
            \F_p[Q,R]/\langle \overline{A}(Q,R) \rangle \xrightarrow{\sim} \overline{M}.
        \end{align*}
        In other words, all relations in $M^p$ derive either from relations in $M$ or from the relation $E_{p-1} \equiv 1 \pmod{p}$.
    \end{enumerate}
\end{proposition}

\subsection{Some facts on primes}

This section lays out elementary results from the analytic study of prime numbers. We do not even require the full strength of the prime number theorem for our purposes, and so we state these results quickly and without much proof. These two lemmas may be seen as consequences, for instance, of work by Rosser and Schoenfeld \cite{RosserSchoenfeld}.

\begin{lemma} \label{L: Rosser's Theorem}
    For $n \geq 1$, we have $p_n > n \log(n)$.
\end{lemma}

\begin{lemma} \label{L: Bertrand}
    For $n \geq 3$, we have $p_n < 2n \log(n)$.
\end{lemma}

Because we have mentioned the central importance the Chinese Remainder Theorem will play in our argument, we will need to multiply together many primes. The next lemma works out a result in this direction using Rosser's theorem.

\begin{definition}
    For integers $n \geq 2$, we let $P_{11}(n) := 11 \cdot 13 \cdot \dots \cdot p_n = \prod_{j=5}^n p_j$.
\end{definition}

One may ask why we do not consider the product over all primes. Our choice is one of simple convenience; this choice does not much change our final result, and this choice avoids technicalities which emerge if one of the lower-weight Eisenstein series $E_2, E_4$, or $E_6$ happens to be 1 modulo $p$.

We prove the following lower bound on $P_{11}(n)$.

\begin{lemma} \label{L: Product bound}
    For $n \geq 2$, we have 
    \begin{align*}
        P_{11}(n) > \lp n \log(n) \rp^n e^{-2n-2}.
    \end{align*}
\end{lemma}

\begin{proof}
    Using Lemma \ref{L: Rosser's Theorem}, we have
    \begin{align*}
        \log P_{11}(n) = \sum_{j=5}^n \log(p_j) > \sum_{j=5}^n \log\lp j \log(j) \rp = \sum_{j=5}^n \log(j) + \sum_{j=5}^n \log(\log(j)).
    \end{align*}
    We now apply several Riemann sum estimates. We note that
    \begin{align*}
        \sum_{j=5}^n \log(j) \geq \int_5^n \log(x)dx = n \log(n) - n + 5 - 5\log(5),
    \end{align*}
    and likewise
    \begin{align*}
        \sum_{j=3}^n \log\log(j) \geq \int_3^n \log\log(x) dx = n \log\log(n) - \li(n) - 5\log\log(5) + \li(5),
    \end{align*}
    where $\li(x)$ is the logarithmic integral
    \begin{align*}
        \li(x) = \int_0^x \dfrac{dt}{\log(t)}.
    \end{align*}
    Using these and other trivial estimates, for instance that $\li(x) \leq x$, we have
    \begin{align*}
        \log P_{11}(n) &> n \log(n) - n + 5 - 5\log(5) + n \log\log(n) - \li(n) - 5\log\log(5) + \li(5) \\
        &> n \log(n) + n \log\log(n) - 2n - 2.
    \end{align*}
    Thus, we conclude by exponentiation that
    \begin{align*}
        P_{11}(n) > \lp n \log(n) \rp^n e^{-2n-2},
    \end{align*}
    which completes the proof.
\end{proof}

\section{Bounds on coefficients of quasimodular forms} \label{Sec: Bounds}

Our method for Sturm bounds will utilize the philosophy that if you determine a finite objects structure modulo sufficiently many primes, you determine its structure over $\Z$. However, the objects with which we are concerned have $q$-coefficients that grow polynomially with $n$, and so ``finitely many" primes depends on how many coefficients are considered. It is therefore of key importance to give explicit constraints on how these coefficients grow. In this section, we use elementary results on convolutions and divisor sums to provide explicit upper bounds which are valid for any product of low-weight Eisenstein series. 

\begin{lemma} \label{L: Convoluion bounds}
    Let $a_n, b_n$ be two sequences with $a_0 = b_0 = 1$, $|a_n| \leq A n^\alpha (1+\log(n))^\gamma$, and $|b_n| \leq B n^\beta \lp 1 + \log(n) \rp^\delta$ for $1 \leq n \leq N$ and $\alpha,\beta,\gamma>0$ integers. Then if $c_n$ is the convolution of $a_n$ and $b_n$, we have
    \begin{align*}
        |c_n| \leq A n^\alpha (1+\log(n))^\gamma + B n^\beta \lp 1 + \log(n) \rp^\delta + \dfrac{AB \alpha! \beta!}{\lp\alpha+\beta+1\rp!} n^{\alpha+\beta+1} (1+\log(n))^{\gamma+\delta}.
    \end{align*}
\end{lemma}

\begin{proof}
    We have for $1 \leq n \leq N$, using an integral comparison, that
    \begin{align*}
        |c_n| \leq \sum_{k=0}^n |a_n b_{k-n}| &\leq |a_n| + |b_n| + AB \sum_{k=0}^n k^{\alpha} (n-k)^\beta \lp 1 + \log(j) \rp^{\gamma} \lp 1 + \log(n-j) \rp^{\delta} \\ &\leq |a_n| + |b_n| + AB n^{\alpha + \beta + 1} (1+\log(n))^{\gamma+\delta} \cdot \dfrac{1}{n} \sum_{k=0}^n \lp\dfrac{k}{n}\rp^\alpha \lp 1-\dfrac{k}{n}\rp^\beta \\ &\leq |a_n| + |b_n| + AB n^{\alpha+\beta+1} (1+\log(n))^{\gamma+\delta} \int_0^1 x^\alpha (1 - x)^\beta dx.
    \end{align*}
    Recall that this integral is precisely the beta function
    \begin{align} \label{Eq: Beta Function}
        B(\alpha,\beta) := \int_0^1 x^{\alpha-1}\lp 1 - x \rp^{\beta-1}dx = \dfrac{\Gamma(\alpha)\Gamma(\beta)}{\Gamma(\alpha+\beta)},
    \end{align}
    where $\Gamma(z)$ is the Euler Gamma function. We therefore obtain
    \begin{align*}
        |c_n| &\leq |a_n| + |b_n| + \dfrac{AB \alpha! \beta!}{\lp\alpha+\beta+1\rp!}  n^{\alpha+\beta+1} (1+\log(n))^{\gamma+\delta} \\ &\leq A n^\alpha (1+\log(n))^{\gamma} + B n^\beta (1+\log(n))^{\delta} + \dfrac{AB \alpha! \beta!}{\lp\alpha+\beta+1\rp!} n^{\alpha+\beta+1} (1+\log(n))^{\gamma+\delta},
    \end{align*}
    which completes the proof.
\end{proof}

\begin{lemma} \label{L: Eisenstein product bound}
    Let the sequence $s_{a,b,c}(n)$ be defined by
    \begin{align*}
        \sum_{n \geq 0} s_{a,b,c}(n) q^n = E_2^a E_4^b E_6^c.
    \end{align*}
    Then we have the bounds
    \begin{align*}
        |s_{a,b,c}(n)| \leq 109 \cdot 2^{2k-1} n^{2k-1} \lp 1 + \log(n) \rp^k.
    \end{align*}
\end{lemma}

\begin{proof}
    Let $\zeta(s) := \sum_{n \geq 1} n^{-s}$ be the Riemann zeta function.
    Using the relationship between $E_{2k}$ and divisor sums, as well as an integral comparison in the case of $E_2 = s_{1,0,0}$, it is elementary to derive the bounds
    \begin{align*}
        |s_{1,0,0}(n)| &\leq 24n\lp 1 + \log(n) \rp, \\
        |s_{0,1,0}(n)| &\leq 240\zeta(3) n^3, \\
        |s_{0,0,1}(n)| &\leq 504\zeta(5) n^5.
    \end{align*}
    In order to use induction, which is our desire, we need to enumerate all bounds up to weight 6; this leaves us to derive bounds for $E_2^2, E_2^3$, and $E_2 E_4$. Now, using Ramanujan's famous differential identities
    \begin{align} \label{Eq: Derivative IDs}
        DE_2 = \dfrac{E_2^2 - E_4}{12}, \ \ \ DE_4 = \dfrac{E_2 E_4 - E_6}{3}, \ \ \ D := q \dfrac{d}{dq},
    \end{align}
    we obtain identities
    \begin{align*}
        s_{2,0,0}(n) &= 240 \sigma_3(n) - 288 n\sigma_1(n), \\
        s_{1,1,0}(n) &= 720 n\sigma_3(n) - 504 \sigma_5(n).
    \end{align*}
    Using the fact that $240n^3 > 288n^2(1+\log(n))$ for $n \geq 2$, we know $s_{2,0,0}(n)$ has positive coefficients beginning at the second, and we may conclude
    \begin{align*}
        |s_{2,0,0}(n)| \leq 240 \zeta(3) n^3.
    \end{align*}
    Similar reasoning shows that $E_2 E_4$ contains only finitely many positive coefficients, and we can conclude after finite computation that
    \begin{align*}
        |s_{1,1,0}(n)| \leq 504 \zeta(5) n^5.
    \end{align*}
    The last base case we must handle is $E_2^3 = E_2 \cdot E_2^2$. We derive thus, from Lemma \ref{L: Convoluion bounds} to conclude
    \begin{align*}
        |s_{3,0,0}(n)| \leq 24n\lp 1 + \log(n) \rp + 240\zeta(3) n^3 + 288\zeta(3) n^5\lp 1 + \log(n) \rp^2 < 576\zeta(3) n^5 \lp 1 + \log(n) \rp^2,
    \end{align*}
    which is valid for $n \geq 1$. This gives, for weights $2,4$, and $6$, the universal bounds
    \begin{align*}
        |s_{a,b,c}(n)| &\leq 24n\lp 1 + \log(n) \rp \ \ \ \text{for all } 2a+4b+6c=2, \\
        |s_{a,b,c}(n)| &\leq 240\zeta(3) n^3 \ \ \ \text{for all } 2a+4b+6c=4, \\
        |s_{a,b,c}(n)| &\leq 576\zeta(3) n^5 \lp 1 + \log(n) \rp^2 \ \ \ \text{for all } 2a+4b+6c=6. 
    \end{align*}
    Observe that the contributors to the maximum bound come from pure powers of $E_2$.

    Our next objective is to take these base cases and establish the explicit upper bound stated in the lemma, namely, that any $a,b,c$ producing a weight $2k$ quasimodular Eisenstein product $E_2^a E_4^b E_6^c$ has coefficients satisfying
    \begin{align*}
        |s_{a,b,c}(n)| \leq C_{\rm max}(2k) n^{2k-1} \lp 1 + \log(n) \rp^k
    \end{align*}
    for $2k \geq 8$.

    To prove this result, any product $E_2^a E_4^b E_6^c$ contains some $a,b,c \not = 0$, and thus each $s_{a,b,c}(n)$ for weight $2k$ is some convolution $(s_{1,0,0} * s_{a-1,b,c})(n)$, $(s_{0,1,0} * s_{a,b-1,c})(n)$, or $(s_{0,0,1} * s_{a,b,c-1})(n)$. We may therefore consider max bounds in these cases only. From the first case, we obtain from Lemma \ref{L: Convoluion bounds} that
    \begin{align} \label{Eq: E_2 bound}
        |s_{a,b,c}(n)| &\leq 24n\lp 1 + \log(n) \rp + C_{\rm max}(2k-2) n^{2k-3} \lp 1 + \log(n) \rp^{k-1} \\ &\hspace{0.5in} + \dfrac{24 C_{\rm max}(2k-2) (2k-3)!}{(2k-1)!} n^{2k-1} \lp 1 + \log(n) \rp^k,
    \end{align}
    from the second case we obtain
    \begin{align} \label{Eq: E_4 bound}
        |s_{a,b,c}(n)| &\leq 240\zeta(3)n^3 + C_{\rm max}(2k-4) n^{2k-5} \lp 1 + \log(n) \rp^{k-2} \\ &\hspace{0.5in} + \dfrac{240\zeta(3) C_{\rm max}(2k-4) (2k-5)!}{(2k-1)!} n^{2k-1} \lp 1 + \log(n) \rp^{k-2},
    \end{align}
    and in the third case by
    \begin{align} \label{Eq: E_6 bound}
        |s_{a,b,c}(n)| &\leq 504\zeta(5)n^5\lp 1 + \log(n) \rp^2 + C_{\rm max}(2k-6) n^{2k-7} \lp 1 + \log(n) \rp^{k-3} \\ &\hspace{0.5in} + \dfrac{504\zeta(5) C_{\rm max}(2k-6) (2k-7)!}{(2k-1)!} n^{2k-1} \lp 1 + \log(n) \rp^{k-1},
    \end{align}

    Each of the three of these must be bounded above by $C_{\rm max}(2k) n^{2k-1} \lp 1 + \log(n) \rp^k$, respectively, since cases with the maximal possible number of powers of $E_2$, $E_4$, or $E_6$, respectively, might be subject only to one of these calculations. Now, from the first derivation, we obtain a required bound
    \begin{align*}
        C_{\rm max}(2k) \geq \dfrac{24}{n^{2k-2} \lp 1 + \log(n) \rp^{k-1}} + \lp \dfrac{1}{n^2\lp 1 + \log(n) \rp} + \dfrac{24 (2k-3)!}{(2k-1)!} \rp \cdot C_{\rm max}(2k-2).
    \end{align*}
    This is a decreasing function of $n$, and so it suffices to substitute $n=1$. After this, we may also note that $\frac{(2k-1)!}{(2k-3)!} \leq \frac{4}{7}$ for $2k \geq 8$, leaving
    \begin{align} \label{Eq: E_2 bound (2)}
        C_{\rm max}(2k) \geq 24 + \dfrac{11}{7} C_{\rm max}(2k-2).
    \end{align}
    We can likewise derive bounds from \eqref{Eq: E_4 bound} and \eqref{Eq: E_6 bound}, which yield
    \begin{align} \label{Eq: E_4 bound (2)}
        C_{\rm max}(2k) \geq 240\zeta(3) + \dfrac{7+2\zeta(3)}{7} C_{\rm max}(2k-4)
    \end{align}
    and
    \begin{align} \label{Eq: E_6 bound (2)}
        C_{\rm max}(2k) \geq 504\zeta(5) + \dfrac{10+\zeta(5)}{10} C_{\rm max}(2k-6).
    \end{align}
    Thus, any determination of $C_{\rm max}(2k)$ that simultaneously satisfies all of \eqref{Eq: E_2 bound (2)}, \eqref{Eq: E_4 bound (2)}, and \eqref{Eq: E_6 bound (2)} will complete the proof. It can be checked that
    \begin{align} \label{Eq: C_max equation}
        C_{\rm max}(2k) = 109 \cdot 2^{2k-1}
    \end{align}
    satisfies these conditions\footnote{The solution $244 (11/7)^{2k-1} - 42$ also works and is smaller; we choose a cleaner expression for convenience and since the choice does not much affect the final result.}.
\end{proof}

We have now focused on a particular basis for $\widetilde{M}_{2k}$ with components having only integer coefficients. However, apart from the case of $\widetilde{M}_2$, this is not a $\Z$-basis; there are forms for instance listed in \eqref{Eq: Derivative IDs} that have integer coefficients but are only $\Q$-combinations in this basis. This is the basis used for Theorem \ref{T: Sturm Q Only}. Therefore, to create a proper $\Z$-basis, one must find a method for identifying all such identities. This will come from the theory of modular forms modulo $p$, which was discussed in Section \ref{S: Modular forms modulo p}. We now apply the classification results in this setting to show how to create $\Z$-basis for $\widetilde{M}_{2k}$, and we show worst-case scenarios for how this affects the order of growth of coefficients of the basis elements.

\begin{proposition} \label{P: Z-basis bound}
    Let $k \geq 10$ be an integer. Then the space $\widetilde{M}_{2k}$ has a $\Z$-basis with the property that for each $f$ in the basis, its Fourier coefficients $a_f(n)$ satisfy the bound
    \begin{align*}
        |a_f(n)| \leq \lp \dfrac{k+6}{6} \rp^{\frac{k^2}{3}} \cdot 109 \cdot 2^{2k-1} n^{2k-1} \lp 1 + \log(n) \rp^k.
    \end{align*}
\end{proposition}

\begin{proof}
    By Lemma \ref{L: Eisenstein product bound}, we know that there exists a basis $\mathcal E_{2k} := \{ E_2^a E_4^b E_6^c \}_{2a+4b+6c=2k}$ of $\widetilde{M}_{2k}$ with integer coefficients. The $\Z$-span of this basis does not cover all of $\widetilde{M}_{2k}^\Z$ in general; but since the $\Z$-module generated by this basis has full rank inside of $\widetilde{M}_{2k}^\Z$, it has finite index inside of $\widetilde{M}_{2k}^\Z$ (for more detail, see e.g. \cite[Theorem 1.17]{StewartTall}). Therefore, there exists some finite number of index-decreasing moves, each of which replaces one element of this with another, that creates an integer basis.

    We perform induction. Suppose by way of a finite number of steps, we have reduced the basis $\mathcal E_{2k}$ to a basis $\mathcal S$, which has integer coefficients and a uniform upper bound function $U_{2k}(n)$ on its Fourier coefficients; that is, $U_{2k}(n)$ is chosen such that for each $g \in \mathcal E_{2k}$, the Fourier coefficients $a_g(n)$ of $g$ are such that $|a_g(n)| \leq U_{2k}(n)$. We suppose now that $f \in \widetilde{M}_{2k}^\Z \backslash \langle \mathcal S \rangle_{\Z}$. We wish to modify $\mathcal S$ to include $f$ in the span. By the fullness of rank of $\langle \mathcal S \rangle_{\Z}$ inside $\widetilde{M}_{2k}^\Z$, there is some integer $m$ such that $mf \in \langle \mathcal S \rangle_{\Z}$; thus we write
    \begin{align*}
        mf = \sum_{g \in \mathcal S} K_{g} g, \ \ \ \text{with} \ \ \ K_{g} \in \Z.
    \end{align*}
    Without loss of generality, we may replace $f$ by some modular form equivalent to $f$ modulo $m$ with the property that $0 \leq K_{a,b,c} < m$. Thus, if $U_{2k}(n)$ is a uniform upper bound on the coefficients of the current basis $\mathcal S$, and if we create a new basis $\mathcal S'$ by finding some $K_{g} \not = 0$ and replacing $g$ by $f$, then the coefficients of $f$ satisfy the bound
    \begin{align*}
        |a_f(n)| \leq \dfrac{\sum_{a,b,c} |K_{a,b,c}| U_{2k}(n)}{m} \leq \dim\lp\widetilde{M}_{2k}\rp U_{2k}(n)
    \end{align*}
    since $0 \leq K_{a,b,c} < m$, and so the new basis has an upper bound $\dim\lp\widetilde{M}_{2k}\rp U_{2k}(n)$. The span of the new basis $\mathcal S'$ includes the span of $\mathcal S$ and $f$, and therefore $[\widetilde{M}_{2k}^\Z : \mathcal S'] < [\widetilde{M}_{2k}^\Z : \mathcal S]$.

    We therefore conclude by induction that if we can demonstrate that some number $N$ of moves of the type above can create a $\Z$-basis of $\widetilde{M}_{2k}^\Z$ from $\mathcal E_{2k}$, then this $\Z$-basis has coefficient-wise uniform upper bound $\lp \dim\widetilde{M}_{2k} \rp^N U_{2k}(n)$. We do this by appeal to  Proposition \ref{P: Serre fact} (3), from which we know the full description of the space of modular forms modulo any prime $p \geq 5$. Firstly, note that a consequence here is that there can be no relations modulo $p$ for modular forms of weight $2k < p - 1$, so we may assume now that $p \leq 2k+1$ if any relation exists.

    Appealing to Proposition \ref{P: Serre fact} and the fact that $E_2$ is a modular form modulo $p$, we have an isomorphism
    \begin{align*}
        \widetilde{M}^p_{2k} \cong M^p_{2k} \cong \F_p[Q,R]/\langle \overline{A}(Q,R) \rangle
    \end{align*}
    where $\overline{A}$ is the reduction modulo $p$ of that bivariate homogeneous polynomial of weight $p-1$ (letting $Q$ and $R$ carry weights 4 and 6) such that $A(E_4,E_6) = E_{p-1}$. Since $E_{p-1} \equiv 1 \pmod{p}$, the multiplication-by-$\overline{A}$ map is injective. We in fact obtain a short exact sequence of $\F_p$-algebras
    \begin{align*}
        0 \rightarrow \F_p[Q,R] \xrightarrow{\times \overline{A}(Q,R)} \F_p[Q,R] \xrightarrow{(Q,R) \mapsto (E_4,E_6)} \bigoplus_{k \geq 0} M_{2k}^p \to 0
    \end{align*}
    from which we will compute dimensions of certain vector spaces. In particular, this short exact sequence is also weight-preserving, in the sense that multiplication by $\overline{A}$ and quotients by $\overline{A}$ have predictable impacts on raising and maintaining weights\footnote{We do not discuss the filtration, or lowest possible weight, of a modular form modulo $p$. We only need results on the realizations of certain quasimodular forms in particular weights. It would be of interest to determine whether filtration theory could improve this argument}, respectively. Since these maps all preserve the grading, where we decompose $\F_p[Q,R] = \bigoplus\limits_{d \geq 0} X_d$ for $X_d$ denoting the subspace of weight $d$ when we assign weights $4, 6$ to $Q, R$, respectively, we obtain by Lemma \ref{L: Quasimodular to modular}, for each $d \geq p-1$,
    \begin{align*}
        0 \rightarrow X_{d-(p-1)} \xrightarrow{\times \overline{A}(Q,R)} X_d \xrightarrow{(Q,R) \mapsto (E_4,E_6)} M^p_d \rightarrow 0.
    \end{align*}
    This short exact sequence of vector spaces therefore implies that for $5 \leq p \leq 2k+1$, we have the identity
    \begin{align*}
        \dim X_{2k-(p-1)} = \dim X_{2k} - \dim M_{2k}^p,
    \end{align*}
    and after observing that $\dim X_{2k} = \dim M_{2k}$, we conclude
    \begin{align*}
        \dim M_{2k-(p-1)} = \dim M_{2k} - \dim M_{2k}^p,
    \end{align*}
    which describes the precise number of independent relations modulo $p$ for the basis $\mathcal E_{2k}$ for $\widetilde{M}_{2k}$. Since no relation can exist for $p > 2k+1$, the number $N$ of independent relations over $\Z$ for the basis $\mathcal E_{2k}$ is bounded above by
    \begin{align*}
        N = \sum_{2 \leq p \leq 2k+1} \lp \dim M_{2k} - \dim M_{2k}^p \rp &\leq 2\dim M_{2k} - \dim M_{2k}^2 - \dim M_{2k}^3 + \sum_{5 \leq p \leq 2k+1} \dim M_{2k-(p-1)},
    \end{align*}
    which after we apply the trivial bound $\pi(2k+1) < k$ and apply Theorem \ref{T: Classical Sturm} to observe that $\dim M_{2k} \leq \frac{k+6}{6}$, we obtain
    \begin{align*}
        N &\leq \dfrac{k+6}{3} + \sum_{5 \leq p \leq 2k+1} \dfrac{\lp k - \frac{p-1}{2} \rp + 6}{6} < \dfrac{k+6}{3} + \dfrac{k(2k+13)}{12} < \dfrac{k^2}{3}
    \end{align*}
    for $k \geq 10$.
    
    Since at most $\frac{k^2}{3}$-many $\Z$-relations exist that can be removed from the basis $\mathcal E_{2k}$ to create a $\Z$-basis for $\widetilde{M}_{2k}$, and since each act of removing one of these $\Z$-relations can only increase the upper bound already established on the basis $\mathcal E_{2k}$ be a factor of $\dim \widetilde{M}_{2k} \leq \frac{k+6}{6}$ all $k$, we conclude that $\widetilde{M}_{2k}$ possesses a $\Z$-basis, where for each $f$ in the basis with coefficients $a_f(n)$, we have
    \begin{align*}
        |a_f(n)| \leq \lp \dfrac{k+6}{6} \rp^{\frac{k^2}{3}} \cdot 109 \cdot 2^{2k-1} n^{2k-1} \lp 1 + \log(n) \rp^k,
    \end{align*}
    which completes the proof.
\end{proof}

\section{Some linear algebra} \label{Sec: Linear Algebra}

The central pillar of our main arguments are results coming out of linear algebra that is used to execute reductions in very large, purely hypothetical Sturm bounds down to reasonable sizes. In this section, we give an important proposition which is required to prove Theorems \ref{T: Main} and \ref{T: Sturm Q Only}.

The first result discusses how, given a non-square matrix with more rows than columns and with bounded entries, one can verify that said matrix has full column rank as a matrix over $\Z$ by verifying that it has full column rank as a matrix over $\mathbb{F}_p$ for a finite number of primes $p$. This result will be applied in the case where each column of said matrix represents some number of initial terms of a basis element for $\widetilde{M}_{2k}$.

We begin the section by stating the well known Hadamard bound on determinants, which is a critical lemma for this section.\footnote{The variation of the Hadamard bound using 2-norms of rows could also be used, but seems not to improve the end result by any order of magnitude.}

\begin{lemma}[{\cite[5.6.P10 on p. 363--364]{MatrixReference}}] \label{L: Hadamard}
    Let $N$ be a $n \times n$ square matrix whose entries are bounded by $B$. Then
    \begin{align*}
        \left| \det(N) \right| \leq B^n n^{n/2}.
    \end{align*}
\end{lemma}

We now begin with the first of the two main results of this section. The purpose of this result is to determine an explicit modulus for which a matrix with full rank over $\Z$ has full rank when reduced\footnote{This is not automatic; consider for instance $\lp \begin{smallmatrix} 1 & 2 \\ 2 & 1 \end{smallmatrix} \rp$ modulo 3.}.

\begin{proposition} \label{P: Full Column Rank}
    Let $A$ be an $m \times n$ matrix with integer entries and with $m \geq n$, each bounded above by a constant $B$. Let $M(B,n) := B^n n^{n/2}$, and choose a constant $C > M(B,n)$. Then $A$ has full column rank if and only if $A$ has full column rank modulo $C$.
\end{proposition}

\begin{remark}
    We make the following remarks:
    \begin{enumerate}
        \item This condition is in fact necessary and sufficient, as the proof clarifies; we only utilize one direction of this implication, and therefore only state this direction.
        \item Our proof will use for $C$ the product of many small primes. In principle prime powers are also permissible; the choice of $C$ is mostly a free choice and can be chosen for convenience.
    \end{enumerate}
\end{remark}

\begin{proof}
    In order for an $m \times n$ non-square matrix to have full column rank, it is necessary and sufficient that at least one of its square $n \times n$ minors is invertible. Thus, it has full column rank if and only if the greatest common divisor among all determinants of its $n \times n$ minors is nonzero. Let this greatest common divisor be given by $G$. Firstly, since $G$ must be less than or equal to each particular $n \times n$ minor determinant in $M$, we have by Lemma \ref{L: Hadamard} that
    \begin{align*}
        G \leq B^n n^{n/2} = M(B,n).
    \end{align*}
    Since $C > M(B,n)$, the statement that $G \not \equiv 0 \pmod{C}$ is equivalent to $G \not = 0$. This completes the proof.
\end{proof}

Proposition \ref{P: Full Column Rank} gives explicit information on how to identify that, for some particular basis, a given number of coefficients is sufficient to establish linear independence of that basis. One might wonder whether a variation of this problem, namely of selecting some $f$ and representing $f = \sum c_i f_i$, might give a more useful method. Such results can be derived by analyzing solutions to equations $\vec{c} = A^+ \vec{f}$, where $A^+$ is the Moore-Penrose inverse \cite[7.3.P7 on p. 453]{MatrixReference} of a non-square matrix, if additional data about $f$ is known in advance. Notably, if $f$ can be shown to have substantially smaller coefficients than a typical quasimodular form of that weight (think of cusp forms for instance) then a notable improvement on Sturm-type bounds can be derived. However, this ingredient in the overall proof does not constitute the ``main term" in our derivation of the Sturm bound. We therefore find it most prudent to only indicate that such results can be derived and to proceed with the most general situation only in this work.

\section{Proof of Theorem \ref{T: Main}} \label{Sec: Proof}

In this section, we prove Theorem \ref{T: Main}. Thus, we let $f \in \widetilde{M}_{2k}$ is some quasimodular form with integer coefficients. Let $\mathcal E_{2k} := \{ E_2^a E_4^b E_6^c \}_{2a+4b+6c=2k}$ be the standard basis of $\widetilde{M}_{2k}$. Relabel these forms as $f_1, f_2, \dots, f_d$ for $d = \dim \widetilde{M}_{2k}$ for convenience; this notation is helpful for assigning each quasimodular form to a column of a matrix.

\subsubsection{Stage 1: From mere existence to an explicit bound}

We begin with the mere consideration that there must exist some Sturm bound. This is true because for each $k \geq 1$, each space $\widetilde{M}_{2k}$ is finite-dimensional, and because each element of $\widetilde{M}_{2k}$ is uniquely determined by its Fourier expansion. Therefore, there must be some integer $m$, which we can assume without loss of generality depends only on $k$ (and not on $f$), so that the first $m$ coefficients of any $f \in \widetilde{M}_{2k}$ determine $f$ uniquely. Now, let $f_i = \sum_{n \geq 0} a_i(n) q^n$ denote the Fourier expansion of each of the $\mathcal E_{2k}$-basis elements $f_i$. We choose some $m$ as determined above, and we consider the matrix
\begin{align*}
    A = \begin{pmatrix}
            a_1(1) & a_2(1) & \cdots & a_d(1) \\
            a_1(2) & a_2(2) & \cdots & a_d(2) \\
            a_1(3) & a_2(3) & \cdots & a_d(3) \\
            \vdots & \vdots & \ddots & \vdots \\
            a_1(m) & a_2(m) & \cdots & a_d(m)
        \end{pmatrix}.
\end{align*}
By our choice of $m$, the matrix $A$ has full column rank. The basic idea of the rest of the proof is to show that the combined leveraging of Proposition \ref{P: Z-basis bound}, Proposition \ref{P: Full Column Rank}, and estimates on the growth of primes, we can create a recursive algorithm by which $m$ can be reduced. That is, we create an algorithm which inputs a full rank matrix whose columns are the Fourier coefficients of a basis of $\widetilde{M}_{2k}$, and which then uses modulo $p$ Sturm bounds for modular forms to prove that many of the rows can be removed without violating the full rank condition. This process can be continued until a stable point is reached.

The steps of the proof will now follow two stages: first, an inductive descent procedure will be implemented which gives an initial explicit upper bound as its output. This initial upper bound will be suboptimal, and so the final stage of the proof will continue to use the algorithm more carefully to reach as close to a natural ``stationary point" of the algorithm as can reasonably be obtained.

\subsubsection{Stage 2: Recursive descent}

We begin now the next stage of argument, by which we demonstrate, by a method reminiscent of Fermat's method of infinite descent, after which that we can assume $m$ is smaller than an explicit function of $k$. This will be executed by first assuming some arbitrarily large upper bound on $m$, then demonstrating that Proposition \ref{P: Full Column Rank} can be leveraged to reduce this upper bound. We then apply this procedure as many times as is required in order to arrive at the claimed bound. These steps are now explained below. We will use the repeated exponential function
\begin{align*}
    \cE_k(x) := k^{k^{k^{\dots}}}, \ \ \ \text{exponentiating $x$ times}.
\end{align*}
Using this very rapidly growing function of $k$, we prove the central result of our argument; that extremely large Sturm bounds can actually be used to construct slightly less (but still extremely large) Sturm bounds. This argument will then bring down $m$ from some unknown upper bound to an extremely large, but known, upper bound. The consequence of this methodology is encapsulated in the following proposition.

\begin{proposition} \label{P: Reduction of Bound}
    The first $k^2 k^k k^{k^k}$ Fourier coefficients of $f \in \widetilde{M}_{2k}^{\Z}$ determine that form uniquely.
\end{proposition}

\begin{proof}
    One can check numerically that this is true for $k < 10$ (see e.g. Section \ref{Sec: Conclusion}) and so we assume for the remainder that $k \geq 10$.
    
    Since quasimodular forms are determined uniquely by their Fourier coefficients, there must be some integer $m$ (depending on $k$) so that each quasimodular form of weight $2k$ is determined uniquely from its first $m$ Fourier coefficients. We assume $m \geq d := \dim \widetilde{M}_{2k}$ without loss of generality.

    Now, suppose that for some integer $x \geq 4$, which can depend on $k$, we have
    \begin{align} \label{Eqn: Arbitary Upper Bound}
        m < k^2 \cE_k(x) \cE_k(x-1).
    \end{align}
    Note that such an integer must exist. We now build the $m \times d$ matrix $A$
    \begin{align*}
    A = \begin{pmatrix}
            a_1(0) & a_2(0) & \cdots & a_d(0) \\
            a_1(1) & a_2(1) & \cdots & a_d(1) \\
            a_1(2) & a_2(2) & \cdots & a_d(2) \\
            a_1(3) & a_2(3) & \cdots & a_d(3) \\
            \vdots & \vdots & \ddots & \vdots \\
            a_1(m) & a_2(m) & \cdots & a_d(m)
        \end{pmatrix},
    \end{align*}
    where the coefficients are chosen according to a $\Z$-basis of $\widetilde{M}_{2k}^\Z$ as guaranteed by Proposition \ref{P: Z-basis bound}. We will apply Proposition \ref{P: Full Column Rank} to this matrix. By Proposition \ref{P: Z-basis bound}, we see that
    \begin{align*}
        B = \lp \dfrac{k+6}{6} \rp^{\frac{k^2}{3}} \cdot 109 \cdot 2^{2k-1} m^{2k-1} \lp 1 + \log(m) \rp^k
    \end{align*}
    must be a uniform upper bound on all elements of the matrix $A$ for $k \geq 10$.

    By Proposition \ref{P: Full Column Rank}, if we can choose a list of primes $p_1, \dots p_N$ for which
    \begin{align*}
        p_1 p_2 \cdots p_N > B^d d^{d/2}
    \end{align*}
    and for which $A$ has full rank modulo $p_j$ for each $j$, then we then know that $A$ has full column rank. We will choose the primes $p_5 = 11, p_6 = 13$, and so on up to $p_N$ for some particular $N$, listing primes in order of size. Then by using Lemma \ref{L: Product bound}, Proposition \ref{P: Full Column Rank} will apply if we simply prove that
    \begin{align} \label{Eqn: First reduction inequality}
        \lp N \log(N) \rp^N e^{-2N} > \lp \dfrac{k+6}{6} \rp^{\frac{dk^2}{3}} \cdot 109^d \cdot 2^{d(2k-1)} m^{d(2k-1)} \lp 1 + \log(m) \rp^{dk} d^{d/2}.
    \end{align}

    To simplify this, we will bound the dimension $d$ of the space $\widetilde{M}_{2k}$, using \eqref{Eq: Quasimodular Dimension} to conclude that $d < \frac{k^2}{9}$ is valid for $k \geq 10$. Therefore, instead of \eqref{Eqn: First reduction inequality} we may prove
    \begin{align*}
        \lp N \log(N) \rp^N e^{-2N-6} > \lp \dfrac{k+6}{6} \rp^{\frac{k^4}{27}} \cdot 109^{\frac{k^2}{9}} \cdot 2^{\frac{2k^3-k^2}{9}} m^{\frac{2k^3-k^2}{9}} \lp 1 + \log(m) \rp^{\frac{k^3}{9}} \lp \dfrac{k^2}{9} \rp^{\frac{k^2}{18}}.
    \end{align*}
    for $k \geq 10$. After taking logarithms and simplifying, we obtain
    \begin{align} \label{Eqn: For later appeal}
        N \log(N) &+ N \log\log(N) - 2N > R_k + \dfrac{2k^3-k^2}{9} \log(m) + \dfrac{k^3}{9} \log\lp 1 + \log(m) \rp,
    \end{align}
    where
    \begin{align} \label{Eqn: R_k point}
        R_k := 6 + \dfrac{k^4}{27} \log\lp\dfrac{k+6}{6}\rp &+ \dfrac{k^2}{18} \log\lp\dfrac{k^2}{9} \rp + \dfrac{k^2}{9} \log(109) + \dfrac{2k^3-k^2}{9} \log(2).
    \end{align}
    It is tedious but elementary to prove from \eqref{Eqn: R_k point} that $R_k < \frac{k^4 \log(k)}{27}$ for $k \geq 10$, and after making other similar trivial simplifications to \eqref{Eqn: For later appeal}, we may prove now the simpler equation
    \begin{align} \label{Eqn: Most Simplified}
        N \log(N) &+ N \log\log(N) - 2N > \dfrac{k^4 \log(k)}{27} + \dfrac{2k^3}{9} \log(m) + \dfrac{k^3}{9} \log\lp 1 + \log(m) \rp.
    \end{align}

    Now, invoking the assumption $m \leq k^2 \cE_k(x) \cE_k(x-1)$, our objective is construct a reasonably small $N$ (permitted to depend on $k$) satisfying
    \begin{align} \label{Eqn: Temp}
        N \log(N) + N \log\log(N) - 2N &> \dfrac{k^4 \log(k)}{27} + \dfrac{2k^3}{9} \log(k^2) + \dfrac{2k^3}{9} \log\cE_k(x) + \dfrac{2k^3}{9} \log\cE_k(x-1) \notag \\ &+ \dfrac{k^3}{9} \log\lp 1 + \log(k^2) + \log\cE_k(x) + \log\cE_k(x-1) \rp.
    \end{align}
    The growth on the right-hand side is dominated by the term $\frac{2k^3}{9} \log\cE_k(x) = \frac{2k^3}{9} \cE_k(x-1) \log(k)$, which suggests a choice of $N = \cE_k(x-1)$. With this choice, we wish to prove \eqref{Eqn: Temp}.
    With the assumptions $x \geq 4$ and $k \geq 10$, this is tedious but straightforward. In particular, the left-hand side has order of growth $N \log(N) \gg \cE_k(x-1) \cE_k(x-2)$ and the right-hand side has growth $k^3 \log\cE_k(x) = k^3 \cE_k(x-1) \log(k)$, and the left-hand side outstrips the right as long as $x \geq 4$. One might for instance rigorously prove separately the inequalities
    \begin{align*}
        N \log(N) > \dfrac{k^4 \log(k)}{27} + \dfrac{2k^3}{9} \log(k^2) + \dfrac{2k^3}{9} \log\cE_k(x) + \dfrac{2k^3}{9} \log\cE_k(x-1)
    \end{align*}
    and
    \begin{align*}
        N \log\log(N) - 2N > \dfrac{k^3}{9} \log\lp 1 + \log(k^2) + \log\cE_k(x) + \log\cE_k(x-1) \rp
    \end{align*}
    under the choice $N = \cE_k(x-1)$ by using elementary inequalities and considerations of monotonicity of various functions involving $\cE_k(x)$ in the $k$ variable. Such a verification is elementary due to the rapid growth of $\cE_k(x)$.

    These arguments, when we backtrack through the inequalities, prove the conclusion of Proposition \ref{P: Full Column Rank}; that is, since $A$ is assumed to have full column rank, we now know full column rank modulo the product of the primes $p$ from 11 to $p_N$.

    We now ask how many of the initial coefficients in $A$ are required to verify the fact that $A$ has full column rank modulo $p$ for primes up to $p_N$, and we will use this calculation in order to apply Proposition \ref{P: Full Column Rank} in the reverse direction. Our objective in doing so is to show using Theorem \ref{T: Classical Sturm} that we can reduce the dimensionality of the matrix $A$  (i.e. reduce $m$).
    
    By Lemma \ref{L: Eisenstein congruences}, every basis element $f_j$ (and therefore also each element of the $\Z$-basis constructed in Proposition \ref{P: Z-basis bound}) is equivalent to a modular form $g_j$ of weight $2pk$. By Theorem \ref{T: Classical Sturm}, it will thus take the first $\frac{kp}{6}+1$ coefficients of a quasimodular form in order to determine that form. That is, the first $\frac{kp}{6}+1$ rows of $A$ are sufficient to determine $A$ modulo $p$. Therefore, the first $\frac{kp_N}{6}+1$ coefficients of $A$ are sufficient to determine $A$ modulo $P_{11}(N)$. By our arguments justifying \eqref{Eqn: First reduction inequality} by proving the stronger inequality \eqref{Eqn: Temp}, then by Proposition \ref{P: Full Column Rank}, it now follows that the first $\frac{kp_N}{6}+1$ coefficients of $A$ are sufficient to determine all of $A$ with the selection $N = \cE_k(x-1)$.

    But by Lemma \ref{L: Bertrand}, we then have that the first
    \begin{align*}
        \dfrac{k p_N}{6}+1 < \dfrac{2 \cE_k(x-1) \log \cE_k(x-1)}{6}+1 < k^2 \cE_k(x-1) \cE_k(x-2).
    \end{align*}
    Therefore, we may, instead of assuming in \eqref{Eqn: Arbitary Upper Bound} that $m \leq k^2 \cE_k(x) \cE_k(x-1)$ coefficients are sufficient to determine $A$ modulo $P_{11}(N)$ as long as $x \geq 4$, and by Proposition \ref{P: Full Column Rank} this is sufficient to prove that $A$ has full rank over $\Z$. Therefore, we can replace our initial assumption of $m \leq k^2 \cE_k(x) \cE_k(x-1)$ with the weaker assumption
    \begin{align} \label{Eqn: New Arbitrary Upper Bound}
        m \leq k^2 \cE_k(x-1) \cE_k(x-2).
    \end{align}

    By assuming the bound \eqref{Eqn: Arbitary Upper Bound} for some $x \gg_k 0$, which must exist because $m$ is necessarily a finite number, we can now prove the bound \eqref{Eqn: New Arbitrary Upper Bound}, as long as $x \geq 4$. Therefore, by induction downward on $x$ (i.e. taking $x \mapsto x-1$ inductively), we can assume \eqref{Eqn: Arbitary Upper Bound} with $x=4$ (dependence on $k$ only affects the number of descent steps required, not the end result), which proves \eqref{Eqn: New Arbitrary Upper Bound} with $x=4$, so that
    \begin{align*}
        m \leq k^2 k^k k^{k^k},
    \end{align*}
    which completes the proof.
\end{proof}

\subsubsection{Stage 3: Final steps of reducing bound}

The final step to prove Theorem \ref{T: Main} is a continued implementation of the very same methodology used to prove Proposition \ref{P: Reduction of Bound}, applied to the starting point $m \leq k^2 k^k k^{k^k}$ given to us by Proposition \ref{P: Reduction of Bound}. Though the original level of refinement didn't allow us to go further down the ladder than this, we can still leverage the same idea to obtain a polynomial bound.

By Proposition \ref{P: Reduction of Bound}, assume that $m \leq k^2 k^k k^{k^k}$ coefficients are sufficient to verify that $A$ has full column rank. We again appeal to Proposition \ref{P: Full Column Rank}. Combining \eqref{Eqn: Most Simplified}, valid for $k \geq 10$, and the same arguments as in the proof of Proposition \ref{P: Reduction of Bound}, we know that $A$ must have full column rank if we can establish that $A$ has full column rank modulo $p$ for the consecutive primes $11, 13, \dots, p_N$ for $N$ satisfying
\begin{align*}
    N \log(N) + N \log\log(N) - 2N > \dfrac{k^4 \log(k)}{27} + \dfrac{2k^3}{9} \log(m) + \dfrac{k^3}{9} \log\lp 1 + \log(m) \rp.
\end{align*}
By using $m \leq k^2 k^k k^{k^k}$, we may prove the stronger inequality
\begin{align*}
    N \log(N) + N \log\log(N) - 2N > \dfrac{k^4 \log(k)}{27} + \dfrac{2k^3}{9} \log\lp k^2 k^k k^{k^k} \rp + \dfrac{k^3}{9} \log\lp 1 + \log\lp k^2 k^k k^{k^k} \rp \rp.
\end{align*}
The term with maximal growth on the right-hand side is $\frac{2k^3}{9} \log\lp k^{k^k} \rp = \frac{2}{9} k^{k+3} \log(k)$, which suggests choosing $N = k^{k+3}$. This choice, by elementary though tedious means similar to those used in the proof of Proposition \ref{P: Reduction of Bound}, holds for $k \geq 10$. Thus, following the same use of Lemma \ref{L: Bertrand} and Theorem \ref{T: Classical Sturm} used in the proof of Proposition \ref{P: Reduction of Bound}, we will only require up to
\begin{align*}
    \dfrac{k p_N}{6} + 1 \leq \dfrac{k^{k+4} \log\lp k^{k+3} \rp}{3}+1 < k^{k+5}
\end{align*}
coefficients. Thus, we may assume $m \leq k^{k+5}$.

We now follow the same procedure. Using \eqref{Eqn: Most Simplified} along with $m \leq k^{k+5}$, we conclude that a sufficient number of primes $11, 13, \dots, p_N$ to utilize Proposition \ref{P: Full Column Rank} is furnished by any $N$ satisfying
\begin{align*}
    N \log(N) + N \log\log(N) - 2N > \dfrac{k^4 \log(k)}{27} + \dfrac{2k^3}{9} \log\lp k^{k+5} \rp + \dfrac{k^3}{9} \log\lp 1 + \log\lp k^{k+5} \rp \rp,
\end{align*}
which it can be shown holds for $N = \lfloor \frac{k^4 \log(k)}{2} \rfloor$ and $k \geq 10$. Thus, using again Lemma \ref{L: Bertrand} and Theorem \ref{T: Classical Sturm} as previously done, we arrive at the fact that the first
\begin{align*}
    \dfrac{k p_N}{6}+1 \leq \dfrac{k^5\log(k) \log\lp \frac{k^4 \log(k)}{2} \rp}{6} + 1 < 2 k^5 \log(k) \log\log(k)
\end{align*}
for $k \geq 10$. After comparing to the cases $1 \leq k \leq 10$, using code discussed in Section \ref{Sec: Conclusion}, the same bound is valid in these cases, and so we may assume $m \leq 2 k^5 \log(k) \log\log(k)$.

We now follow the same procedure for a final instance. Using \eqref{Eqn: Most Simplified} along with $m \leq 2 k^5 \log(k) \log\log(k)$, we conclude that a sufficient number of primes $11, 13, \dots, p_N$ to utilize Proposition \ref{P: Full Column Rank} is furnished by any $N$ satisfying
\begin{align*}
    N \log(N) + N \log\log(N) - 2N &> \dfrac{k^4 \log(k)}{27} + \dfrac{2k^3}{9} \log\lp 2 k^5 \log(k) \log\log(k) \rp \\ &+ \dfrac{k^3}{9} \log\lp 1 + \log\lp 2 k^5 \log(k) \log\log(k) \rp \rp,
\end{align*}
which it can be shown holds for $N = \lfloor \frac{k^4}{16} \rfloor$ and $k \geq 10$.\footnote{One might choose a mildly better constant. But since order of magnitude stabilizes at this point, and since this bound is not theoretical optimal in any case, we do not find that such strain is worthwhile.} Thus, using again Lemma \ref{L: Bertrand} and Theorem \ref{T: Classical Sturm} as previously done, we arrive at the fact that the first
\begin{align*}
    \dfrac{k p_N}{6}+1 \leq \dfrac{k^5 \log(k^4/16)}{48} + 1 < \dfrac{k^5 \log(k)}{12}
\end{align*}
for $k \geq 10$. After comparing to the cases $1 \leq k \leq 10$, using code discussed in Section \ref{Sec: Conclusion}, the same bound is valid in these cases, and this proves Theorem \ref{T: Main} (1); since we have made use of a $\Z$-basis for $\widetilde{M}_{2k}^\Z$, no relations among its elements modulo $m$ for any $m \geq 2$ may exist, and therefore this bound also extends to prove Theorem \ref{T: Main} (2). $\hfill\qed$

\section{Proof of Theorem \ref{T: Sturm Q Only}} \label{S: Proof Q}

In this section, we prove Theorem \ref{T: Sturm Q Only}, and so $f \in \widetilde{M}_{2k}$ is some quasimodular form with integer coefficients. Let $\mathcal E_{2k} := \{ E_2^a E_4^b E_6^c \}_{2a+4b+6c=2k}$ be the standard basis of $\widetilde{M}_{2k}$. Relabel these forms as $f_1, f_2, \dots, f_d$ for $d = \dim \widetilde{M}_{2k}$ for convenience; this notation is helpful for assigning each quasimodular form to a column of a matrix.

\subsubsection{Stage 1: From mere existence to an explicit bound}

The proof follows in the same steps as Stage 1 in the proof of Theorem \ref{T: Main}. There must be some Sturm bound because for each $k \geq 1$, each space $\widetilde{M}_{2k}$ is finite-dimensional, and because each element of $\widetilde{M}_{2k}$ is uniquely determined by its Fourier expansion. Therefore, there must be some integer $m$, which we can assume without loss of generality depends only on $k$ (for instance by Theorem \ref{T: Main}), so that the first $m$ coefficients of any $f \in \widetilde{M}_{2k}$ determine $f$ uniquely. Now, let $f_i = \sum_{n \geq 0} a_i(n) q^n$ denote the Fourier expansion of each of the $\mathcal E_{2k}$-basis elements $f_i$. We choose some $m$ as determined above, and we consider the matrix
\begin{align*}
    A = \begin{pmatrix}
            a_1(0) & a_2(0) & \cdots & a_d(0) \\
            a_1(1) & a_2(1) & \cdots & a_d(1) \\
            a_1(2) & a_2(2) & \cdots & a_d(2) \\
            a_1(3) & a_2(3) & \cdots & a_d(3) \\
            \vdots & \vdots & \ddots & \vdots \\
            a_1(m) & a_2(m) & \cdots & a_d(m)
        \end{pmatrix}.
\end{align*}
By our choice of $m$, the matrix $A$ has full column rank. The basic idea of the rest of the proof follows precisely that used to prove Theorem \ref{T: Main}. We now execute the same algorithm.

\subsubsection{Stage 2: Recursive descent}

We begin now the next stage of argument, by which we demonstrate that we can assume $m$ is smaller than an explicit function of $k$. We will again use the repeated exponentials
\begin{align*}
    \cE_k(x) := k^{k^{k^{\dots}}}, \ \ \ \text{exponentiating $x$ times}.
\end{align*}
We prove the following analog of Proposition \ref{P: Reduction of Bound}

\begin{proposition} \label{P: Q - Reduction of Bound}
    There is an integer $m \leq k^{k+3}$ such that the first $m$ coefficients of a quasimodular form $f \in \widetilde{M}_{2k}$ determine that form uniquely.
\end{proposition}   

\begin{proof}
    The bound can be checked numerically for $k < 10$, and therefore we may assume that $k \geq 10$.
    
    As discussed before, there must be some integer $m$ satisfying this property since quasimodular forms make up a finite dimensional vector space, each element of which is uniquely determined by its Fourier coefficients.
    
    Now, suppose that for some integer $x \geq 3$, which is allowed to depend on $k$, we have
    \begin{align} \label{Eqn: Arbitary Upper Bound *}
        m < k^2 \cE_k(x) \cE_k(x-1).
    \end{align}
    Note that such an integer must exist. We now build the $m \times d$ matrix $A$
    \begin{align*}
        A = \begin{pmatrix}
                a_1(0) & a_2(0) & \cdots & a_d(0) \\
                a_1(1) & a_2(1) & \cdots & a_d(1) \\
                a_1(2) & a_2(2) & \cdots & a_d(2) \\
                a_1(3) & a_2(3) & \cdots & a_d(3) \\
                \vdots & \vdots & \ddots & \vdots \\
                a_1(m) & a_2(m) & \cdots & a_d(m)
            \end{pmatrix}.
    \end{align*}
    to which we will apply Proposition \ref{P: Full Column Rank}. By Lemma \ref{L: Eisenstein product bound}, we obtain as a valid upper bound
    \begin{align*}
        B = 109 \cdot 2^{2k-1} m^{2k-1} \lp 1 + \log(m) \rp^k.
    \end{align*}
    By Proposition \ref{P: Full Column Rank}, if we can choose a list of primes $p_1, \dots p_N$ for which
    \begin{align*}
        p_1 p_2 \cdots p_N > B^d d^{d/2}
    \end{align*}
    and for which $A$ has full rank modulo $p_j$ for each $j$, then we then know that $A$ has full column rank. We will choose the primes $p_2 = 3, p_3 = 5$, and so on up to $p_N$ for some particular $N$, listing primes in order of size. Then by using Lemma \ref{L: Product bound}, Proposition \ref{P: Full Column Rank} will apply if we simply prove that
    \begin{align} \label{Eqn: First reduction inequality *}
        \lp N \log(N) \rp^N e^{-2N} > \lp 109 \cdot 2^{2k-1} m^{2k-1} \lp 1 + \log(m) \rp^k \rp^d d^{d/2}.
    \end{align}
    To simplify this, we use $d < \frac{k^2}{9}$ for $k \geq 10$, which follows from \eqref{Eq: Quasimodular Dimension}. Therefore, instead of \eqref{Eqn: First reduction inequality *} we may prove
    \begin{align} \label{Eqn: Second reduction inequality *}
        \lp N \log(N) \rp^N e^{-2N-6} > 109^{k^2/9} \cdot 2^{k^2(2k-1)/9} m^{k^2(2k-1)/9} \lp 1 + \log(m) \rp^{k^3/9} \lp \dfrac{k^2}{9} \rp^{k^2/18}.
    \end{align}
    After taking logarithms and simplifying, we obtain
    \begin{align} \label{Eqn: For later appeal *}
        N \log(N) &+ N \log\log(N) - 2N > R_k + \dfrac{k^2(2k-1)}{9} \log(m) + \dfrac{k^3}{9}\log\lp 1 + \log(m) \rp,
    \end{align}
    where
    \begin{align} \label{Eqn: R_k point *}
        R_k := 6 + \dfrac{k^2}{9} \log(109) + \dfrac{k^2(2k-1)}{9} \log(2) + \dfrac{k^2}{18} \log\lp \dfrac{k^2}{9} \rp.
    \end{align}
    It can be proved with elementary techniques that $R_k < \frac{2k^3}{9}$ for $k \geq 10$, and when we additionally use \eqref{Eqn: Arbitary Upper Bound} and other simplifications, we aim to prove that
    \begin{align*}
        N \log(N) + N \log\log(N) - 2N > \dfrac{2k^3}{9} + \dfrac{2k^3}{9} \log\lp k^2 \cE_k(x) \cE_k(x-1) \rp + \dfrac{k^3}{9} \log\lp 1 + \log\lp k^2 \cE_k(x) \cE_k(x-1) \rp \rp.
    \end{align*}
    Now, the right-hand side has dominating term approximately $k^3 \log \cE_k(x) = k^3 \log(k) \cE_k(x-1)$, and so we choose $N = \cE_k(x-1)$ to reflect this; we thus need to verify
    \begin{align} \label{Eqn: Third reduction inequality *}
        \cE_k(x-1) \log&\lp \cE_k(x-1) \rp + \cE_k(x-1) \log\log\lp \cE_k(x-1) \rp - 2 \cE_k(x-1) \notag \\ &> \dfrac{2k^3}{9} + \dfrac{2k^3}{9} \log\lp k^2 \cE_k(x) \cE_k(x-1) \rp + \dfrac{k^3}{9} \log\lp 1 + \log\lp k^2 \cE_k(x) \cE_k(x-1) \rp \rp.
    \end{align}
    We now show this holds under the assumption $k \geq 10$. Diving both sides by $\cE_k(x-1)$, we see that
    \begin{align*}
        \log \cE_k(x-1) +& \log\log \cE_k(x-1) - 2 \\ &> \dfrac{2k^3}{9\cE_k(x-1)} + \dfrac{2k^3 \log\lp k^2 \cE_k(x) \cE_k(x-1) \rp}{9\cE_k(x-1)} + \dfrac{k^3 \log\lp 1 + \log\lp k^2 \cE_k(x) \cE_k(x-1) \rp \rp}{9\cE_k(x-1)}.
    \end{align*}
    This is done by proving separately that
    \begin{align*}
        \log \cE_k(x-1) > \dfrac{2k^3 \log\lp k^2 \cE_k(x) \cE_k(x-1) \rp}{9\cE_k(x-1)} \sim \dfrac{2k^3\log(k)}{9},
    \end{align*}
    which is true since $\cE_k(x-1) \geq k^k$ for $x \geq 3$, and
    \begin{align*}
        \log\log \cE_k(x-1) > 2 + \dfrac{2k^3}{9\cE_k(x-1)} + \dfrac{k^3 \log\lp 1 + \log\lp k^2 \cE_k(x) \cE_k(x-1) \rp \rp}{9\cE_k(x-1)} \sim 2,
    \end{align*}
    in which we use $\log\log \cE_k(x-1) = \log\lp \cE_k(x-2) \log(k) \rp \gg \log(k)$ for $x \geq 3$ and $k \geq 10$. These asymptotics verify \eqref{Eqn: Third reduction inequality *} for $k \gg 0$; the choice $k \geq 10$ can then be demonstrated using considerations of monotonicity.

    These arguments, when we backtrack through the inequalities, prove the conclusion of Proposition \ref{P: Full Column Rank}; that is, any $A$ which has full column rank modulo the primes $p$ from 11 to $p_N$, must have full column rank.

    We now ask how many of the initial coefficients in $A$ are required to verify the fact that $A$ has full column rank modulo $p$ for primes up to $p_N$. By Lemma \ref{L: Eisenstein congruences}, every basis element $f_j$ is equivalent to a modular form $g_j$ of weight $2pk$, and by Proposition \ref{P: Serre fact}, these $g_j$ are pairwise inequivalent modulo primes $p \geq 11$. By Theorem \ref{T: Classical Sturm} in the modulo $p$ case, it will take the first $\frac{kp}{6}+1$ coefficients of a quasimodular form in order to determine that form. That is, the first $\frac{kp}{6}+1$ rows of $A$ are sufficient to determine $A$ modulo $p$. Therefore, the first $\frac{kp_N}{6}+1$ coefficients of $A$ are sufficient to determine $A$ modulo $P_{11}(N)$. By our arguments justifying \eqref{Eqn: First reduction inequality *} by proving the stronger inequality \eqref{Eqn: Third reduction inequality *}, then by Proposition \ref{P: Full Column Rank}, it now follows that the first $\frac{kp_N}{6}+1$ coefficients of $A$ are sufficient to determine all of $A$. But by Lemma \ref{L: Bertrand}, we then have that the first
    \begin{align*}
        \dfrac{kp_N}{6}+1 < \dfrac{2 \cE_k(x-1) \log \cE_k(x-1)}{6}+1 < k^2 \cE_k(x-1) \cE_k(x-2).
    \end{align*}
    Therefore, we may, instead of assuming in \eqref{Eqn: Arbitary Upper Bound *} that $m \leq k^2 \cE_k(x) \cE_k(x-1)$ coefficients are sufficient to determine $A$ as long as $x \geq 3$, we can instead assume that
    \begin{align} \label{Eqn: New Arbitrary Upper Bound *}
        m \leq k^2 \cE_k(x-1) \cE_k(x-2)
    \end{align}
    coefficients of $A$ are sufficient to determine $A$.

    We now perform induction on $x \geq 3$. By assuming the bound \eqref{Eqn: Arbitary Upper Bound *} for some $x \gg_k 0$, which must exist because $m$ is finite, we can prove the bound \eqref{Eqn: New Arbitrary Upper Bound *}, as long as $x \geq 3$. Therefore, by induction downward on $x$ (i.e. taking $x \mapsto x-1$ inductively, we can assume \eqref{Eqn: Arbitary Upper Bound *} with $x=3$, which proves \eqref{Eqn: New Arbitrary Upper Bound *} with $x=3$, so that
    \begin{align*}
        m \leq k^{k+3},
    \end{align*}
    which completes the proof.
\end{proof}

\subsubsection{Stage 3: Final steps of reducing bound}

The final step to prove Theorem \ref{T: Main} is a more refined implementation of the very same methodology used to prove Proposition \ref{P: Q - Reduction of Bound}, applied to the starting point $m \leq k^{k+3}$ given to us by Proposition \ref{P: Q - Reduction of Bound}. Though the original level of refinement didn't allow us to go further down the ladder than this (this would end up implying that $\widetilde{M}_{2k}$ has a Sturm bound that is linear in $k$, which by its dimensionality is false), we can still leverage the same idea with extra care to obtain a polynomial bound.

So, by Proposition \ref{P: Reduction of Bound}, assume that $m \leq k^{k+3}$ coefficients are sufficient to verify that $A$ has full column rank. We again appeal to Proposition \ref{P: Full Column Rank}. Combining \eqref{Eqn: For later appeal}, trivial inequalities including $R_k < \frac{2k^3}{9}$ for $k \geq 10$, and the same arguments as in the proof of Proposition \ref{P: Reduction of Bound}, we know that $A$ must have full column rank if we can establish that $A$ has full column rank modulo $p$ for the consecutive primes $11, 13, \dots, p_N$ for $N$ satisfying
\begin{align} \label{Eqn: Modified log inequality}
    N \log(N) + N \log\log(N) - 2N > \dfrac{2k^3}{9} + \dfrac{2k^3}{9} \log(m) + \dfrac{k^3}{9} \log\lp 1 + \log(m) \rp.
\end{align}
Since we have established by Proposition \ref{P: Reduction of Bound} that $m \leq k^{k+3}$ may be assumed without loss of generality, we may therefore resolve instead
\begin{align*}
    N \log(N) + N \log\log(N) - 2N > \dfrac{2k^3}{9} + \dfrac{2k^3}{9}(k+3) \log(k) + \dfrac{k^3}{9} \log\lp 1 + (k+3)\log(k) \rp.
\end{align*}
With the selection $N = \lfloor \frac{k^4}{3} \rfloor$, this holds for $k \geq 10$. Therefore, $A$ has full column rank as long as we can verify it has full column rank modulo primes starting at 11 and up to $p_N$. Using Lemmas \ref{L: Rosser's Theorem} and \ref{L: Bertrand}, this prime is bounded by $\frac{8}{3} k^4 \log(k)$, and then by Theorem \ref{T: Classical Sturm} we note that in order to verify that $A$ has full column rank modulo all primes up to $p_N \leq \frac{8}{3} k^4 \log(k)$, one needs only check the first
\begin{align*}
    \dfrac{k p_N}{6}+1 \leq \dfrac{4 k^5 \log(k)}{9}
\end{align*}
coefficients of each form. Thus, we may assume in place of $m \leq k^{k+3}$ that $m \leq \frac{4 k^5 \log(k)}{9}$. Under this new assumption, and again applying \eqref{Eqn: Modified log inequality}, we must find $N$ satisfying
\begin{align} \label{Eqn: Final log inequality}
    N \log(N) + N \log\log(N) - 2N > \dfrac{2k^3}{9} + \dfrac{2k^3}{9} \log\lp \dfrac{4 k^5 \log(k)}{9} \rp + \dfrac{k^3}{9} \log\lp 1 + \log\lp \dfrac{4 k^5 \log(k)}{9} \rp \rp.
\end{align}
The selection $N = \lfloor \frac{5k^3}{9} \rfloor$ is easily verified to hold for $k \geq 10$, and following again the uses of Proposition \ref{P: Full Column Rank}, Lemmas \ref{L: Rosser's Theorem} and \ref{L: Bertrand}, and Theorem \ref{T: Classical Sturm}, we establish that the last prime $p_N$ which we must check satisfies $p_N \leq 2N \log(N) \leq \frac{10k^3}{9} \log\lp \frac{5k^3}{9} \rp$, and so we need only check
\begin{align*}
    \dfrac{k p_N}{6}+1 \leq \dfrac{5 k^4 \log\lp \frac{5k^3}{9} \rp}{27} + 1 < \dfrac{5 k^4 \log\lp k \rp}{9}
\end{align*}
coefficients of each form before we know that $A$ has full column rank. So, we may assume $m \leq \frac{5 k^4 \log\lp k \rp}{9}$.

Repeating our use of \eqref{Eqn: Modified log inequality}, we can use primes starting at 11 up to any $p_N$ with $N$ satisfying
\begin{align} \label{Eqn: Final Final log inequality}
    N \log(N) + N \log\log(N) - 2N > \dfrac{2k^3}{9} + \dfrac{2k^3}{9} \log\lp \dfrac{5 k^4 \log\lp k \rp}{9} \rp + \dfrac{k^3}{9} \log\lp 1 + \log\lp \dfrac{5 k^4 \log\lp k \rp}{9} \rp \rp.
\end{align}
This holds for $k \geq 10$ with $N = \lfloor \frac{k^3}{2} \rfloor$, leads as before to a prime choice of $p_N \leq k^3 \log\lp \frac{k^3}{2} \rp < 3 k^3 \log(k)$, and so we may now assume
\begin{align*}
    m \leq \dfrac{kp_N}{6} + 1 \leq \dfrac{3k^4 \log(k^3)}{6}+1 < \dfrac{k^4 \log(k)}{2} + 1.
\end{align*}
This completes the proof for $k \geq 10$, and so the result follows. $\hfill\qed$

\begin{remark}
    It is worthwhile to note that in the final application, we are not able to improve on $N = O(k^3)$, and so the final result cannot, by these methods alone, improve on a Sturm bound on order $O\lp k^4 \log(k) \rp$. Leading constants can also not be improved by any substantial amount without revising some of the inputs (i.e. improving upon the Hadamard bound for the specific determinants under consideration).
\end{remark}

\section{Concluding Remarks} \label{Sec: Conclusion}

\subsection{Why these bounds?}

We emphasize here previous remarks which have been made as to the direct source of the bounds appearing in Theorems \ref{T: Main} and \ref{T: Sturm Q Only}. The limiting factors of these proofs arise out of Proposition \ref{P: Z-basis bound} in the first case, and Proposition \ref{P: Full Column Rank} (specifically, the determinant bound used therein) in the second case. In each proof, as has been mentioned in due course, we have extended our recursive descent until we have arrived as the optimal polynomial power in $k$ which is permitted by the method, and we have made reasonable efforts to refine constants. For future work, it should therefore be considered paramount to improve upon Propositions \ref{P: Z-basis bound} by either constructing a suitable $\Z$-basis, or constructing a different basis with integer coefficients having less need to remove relations modulo $m$, or else by some other improvement to the application of Swinnerton-Dyer's theorem \cite{SwinnertonDyer}. As for the improvement of Proposition \ref{P: Full Column Rank}, one may ask whether, in the particular cases we consider, one may derive bounds on determinants far superior to the Hadamard bound (almost certainly this is true, but the author is unsure how difficult such a task might be). Improvement to Lemma \ref{L: Eisenstein product bound} is not possible in any degree that will be helpful because growth rates of divisor sums are known and the growth of the Bernoulli numbers $B_{2k}$ does not permit significantly better dependence on $k$.

It may be possible to improve on the use of primes and Theorem \ref{T: Classical Sturm} by permitting the use of prime powers in building the modulus used in Proposition \ref{P: Full Column Rank}. We leave such considerations, and whether they yield substantially improved or only moderately improved bounds, as questions for future work.

\subsection{Why not higher levels?}

We give here further details on the remark made after Theorem \ref{T: Main}. In the classical paper of Sturm \cite{Sturm}, an argument is provided which immediately lifts the Sturm bound for level one to a Sturm bound for any congruence subgroup $\Gamma$ by consideration of the product form
\begin{align*}
    F = \prod_{\gamma \in \Gamma\backslash\SL_2(\Z)} f|_k\gamma
\end{align*}
where $|_k$ is the usual weight $k$ slash action and $f$ is taken to be a quasimodular form on $\Gamma$ only. One can run most of Sturm's argument in the same method as he does; namely, one can show that, up to bounded denominators which can be removed, $F$ is a quasimodular form on $\SL_2(\Z)$ of weight $mk$, where $m$ is the index of $\Gamma$ in $\SL_2(\Z)$, and with integer Fourier expansion at $\infty$. If $f$ is supposed to have $\ord_\infty(f) > v$, then $F$ has order of vanishing $\ord_\infty(F) > mv$. The classical Sturm argument relies on the fact that the new weight $mk$, the new order of vanishing $mv$. Because the Sturm bound of Theorem \ref{T: Classical Sturm} is linear in $k$, one may then apply Theorem \ref{T: Classical Sturm} by substituting $k \mapsto mk$ for the weight without issue. However, the Sturm bounds for quasimodular forms, even a conjecturally optimal one, must grow at least quadratically in $k$, and therefore this substitution argument does not quite succeed. While it is probably still true that the true Sturm bound for $\Gamma$ should be linear in $m$ and quadratic in $k$ in our opinion, the old technique of Sturm is not sufficient to prove this. We therefore leave the study of Sturm bounds for cases of higher level for a continuation of this work.

\subsection{Some Small Weights}

To aid in motivation of future work in this area, it is instructive to investigate some small examples which help in understanding the difficulties of listing precise $\Z$-bases for $\widetilde{M}^\Z_{2k}$. The main obstruction to a more general approach, we note, is the inability to predict precisely how to identify all identities of the sort we see in \eqref{Eq: Derivative IDs}.

\subsubsection{Weight 2}

This case is trivial; the only quasimodular forms of weight 2 are multiples of $E_2$; thus $\{ E_2 \}$ is a $\Z$-basis for $\widetilde{M}_2^\Z$ and a single coefficient is sufficient to determine the form.

\subsubsection{Weight 4}

The standard basis for $\widetilde{M}_4$ is $\{ E_2^2, E_4 \}$. This basis is not a $\Z$-basis, however, since $E_2^2 \equiv E_4 \pmod{12}$. Following the method used in Proposition \ref{P: Z-basis bound}, we may consider instead the basis $\{ E_4, \frac{E_2^2 - E_4}{12} \}$. We get $q$-series expansions $E_4 = 1 + 240 q + \dots$ and $\frac{E_2^2 - E_4}{12} = -24q + \dots$. This is a triangular $\Z$-basis, and thus the first two coefficients suffice to compute a representation for any element of $\widetilde{M}_4^\Z$ using this basis.

Other $\Z$-bases can be found as well. For instance, the elements
\begin{align*}
    \dfrac{1}{6} E_4 + \dfrac{5}{6} E_2^2 &= 1 + 0q + 720q^2 + 3840q^3 + \cdots \\
    \dfrac{7}{6} E_4 - \dfrac{1}{6} E_2^2 &= 1 + 244q + 2448q^2 + 7296q^3 + \cdots
\end{align*}
also form a $\Z$-basis, though not a triangular one as in the previous example.

\subsubsection{Weight 6}

The standard basis for $\widetilde{M}_6$ is $\{ E_2^3, E_2 E_4, E_6 \}$. This basis is not a $\Z$-basis; many relations modulo various small primes can be established. After some numerical work, one can find that the three quasimodular forms
\begin{align*}
    E_2^3 &= 1 - 72q + 1512 q^2 - 3744 q^3 + \dots, \\
    \dfrac{E_2^3 - E_2 E_4}{288} &= 0 + 1q - 18q^2 - 204q^3 + \dots, \\
    \dfrac{5 E_2^3 - 3 E_2 E_4 - 2 E_6}{51640} &= 0 + 0q + 1q^2 + 8q^3 + \dots.
\end{align*}
form a triangular basis with integer coefficients and thus generate $\widetilde{M}_6^\Z$. Thus, three is a Sturm bound for weight 6. As in the previous example, other $\Z$-bases may be constructed, for instance
\begin{align*}
    \dfrac{1}{120} E_6 + \dfrac{21}{80} E_2 E_4 + \dfrac{35}{48} E_2^3 &= 1 + 0q + 0q^2 - 20160q^3 + \cdots \\
    -\dfrac{59}{120} E_6 + \dfrac{161}{80} E_2 E_4 - \dfrac{25}{48} E_2^3 &= 1 + 720q + 0q^2 - 63360q^3 + \cdots \\
    -\dfrac{83}{120} E_6 + \dfrac{57}{80} E_2 E_4 + \dfrac{47}{48} E_2^3 &= 1 + 432q + 10368q^2 + 36864q^3 + \cdots
\end{align*}
is also a $\Z$-basis.

\subsubsection{Coding and observations in some larger weights}

As discussed above, one of the major obstructions to Sturm-type bounds for quasimodular forms has been the absence of any known $\Z$-basis for these spaces. The standard basis generated by the products $E_2^a, E_4^b E_6^c$ of a particular weight, as we have said, generates only a finite index submodule of $\widetilde{M}_{2k}^\Z$. As part of Proposition \ref{P: Z-basis bound}, we described the following algorithm for producing from this standard basis a $\Z$-basis:
\begin{enumerate}
    \item Form a matrix $M$ with the $q$-series expansions of each $E_2^a E_4^b E_6^c$ of a given weight $2k$ as its columns, and with sufficiently many rows to establish linear independence of the columns.
    \item For each prime $2 \leq p \leq 2k+1$:
    \begin{enumerate}
        \item Check whether the columns are linearly dependent modulo $p$:
        \item If such a dependence exists, say $\sum c_i f_i \equiv 0 \pmod{p}$ for some basis elements $f_i$, form a new element $f = \frac{\sum c_i f_i}{p^M} \in \Z[[q]]$ by taking the largest $M$ such that $p^M$ divides every coefficient of $f$;
        \item Perform column operations on $M$ corresponding to the replacement of $f_i$ by $f$ in the basis;
        \item Repeat until the columns of $M$ linearly independent modulo $p$.
    \end{enumerate}
\end{enumerate}
We noted in the proof of Proposition \ref{P: Z-basis bound} that this process necessarily terminates in finite time, and that the particular range of primes cited is sufficient to guarantee that the resulting basis is a $\Z$-basis. We give in the Appendix the precise code used by the author to generate $\Z$-basis using calculations for the Hermite normal form of matrices.

Comments on the observed output of the code are significant. The author has run values $2 \leq k \leq 20$, and in each case the $\Z$-basis only required a number of coefficients equal to the dimension of the space in question\footnote{The basis construction $E_2^a E_4^b E_6^c$ with $a+2b+3c=k$ for the space $\widetilde{M}_{2k}$ would imply that the number of partitions into parts at most 3, or at most three distinct part sizes, gives these coefficients. This sequence can be found in OEIS entry A001399.} in order to verify the $\Z$-basis; that is, it appears that the theoretically optimal Sturm bound is in fact true. We therefore leave this as a conjecture:

\begin{conjecture}
    Let $k \geq 1$ be an integer, and let $f \in \widetilde{M}_{2k}$. Then $f$ is determined uniquely by its first $\dim \widetilde{M}_{2k}$ Fourier coefficients. Furthermore, $f \pmod{m}$ for any $m \geq 2$ is uniquely determined by its first $\dim \widetilde{M}_{2k}$ coefficients modulo $m$.
\end{conjecture}

If the conjecture is true, the work of van Ittersum and Ringeling \cite{IttersumRingeling} in the case of depth 1 (i.e. quasimodular forms in which $E_2^2$ never appears) suggests that next steps towards its proof would be to developed the right technology to generalize the proof of the classical Sturm bound using the valence formula rather than the proof using an explicit triangular basis. However, a general formula describing a triangular $\Z$-basis for $\widetilde{M}_{2k}^\Z$ would be of incredible interest as well.

\subsection{Possible Next Directions}

In this paper, we have used no special analytic features of the transformation law for quasimodular forms. Indeed, the core assumptions of the method are almost entirely arithmetical. The technique developed in this paper yields the following general principle: 

\begin{principle}
    Given a family $\mathcal F^*$ of $q$-series which is congruent modulo infinitely many primes $p$ to another family of $q$-series $\mathcal F$ possessing a Sturm bound, and where coefficients in the $q$-series expansions for $\mathcal F^*$ can be reasonably bounded, one can lift the Sturm bound for $\mathcal F$ into a Sturm bound for $\mathcal F^*$.
\end{principle}

It is quite natural to seek extensions of this general principle. One generalization, of course, would be to prove results for quasimodular forms with level structure; this will be pursued by the author and collaborators in upcoming work. Other natural extensions may include Hecke triangle groups \cite{Henry} or other quasi-objects in modular form theory \cite{KanekoZagier}. Another of the first and most natural extensions of the principle in a nontrivial way is to the case of mixed weight quasimodular forms, i.e. elements of the algebra $\bigoplus_{k \geq 0} \widetilde{M}_{2k}$ which do not belong to any graded component. Many applications of quasimodular forms naturally produce such objects. For brevity, we discuss one example that arises in many of these contexts; the $q$-generalizations of multiple zeta values, which we simply call $q$-multiple zeta values. The classical starting point for this study MacMahon's work on generalized divisor sums \cite{MacMahon}. These sums are defined by the $q$-series
\begin{align*}
    \mathcal U_a(q) := \sum_{0 < m_1 < m_2 < \dots < m_a} \dfrac{q^{m_1 + m_2 + \dots + B}}{\lp 1 - q^{m_1} \rp^2 \lp 1 - q^{m_2} \rp^2 \cdots \lp 1 - q^{m_a} \rp^2}.
\end{align*}
It was conjectured by MacMahon and proven by Andrews and Rose \cite{AndrewsRose} that the $\mathcal U_a$ are mixed weight quasimodular forms. In fact, MacMahon's functions belong to a much wider universe of $q$-multiple zeta values, which are natural $q$-generalizations of the multiple zeta values
\begin{align*}
    \zeta\lp k_1, \dots, k_a \rp := \sum_{0 < n_1 < n_2 < \dots < n_a} \dfrac{1}{n_1^{k_1} n_2^{k_2} \cdots n_a^{k_a}}.
\end{align*}
These real numbers have appeared frequently across mathematics and physics; their modern study was initiated by Hoffman \cite{Hoffman}, and the literature on relations between these values is extensive. These $q$-multiple zeta values take the general shape
\begin{align*}
    \mathcal{U}_{a}\lp P_1, \dots, P_a; k_1, \dots, k_a \rp := \sum_{0 < m_1 < m_2 < \dots < m_a} \dfrac{P_1(q^{m_1}) \cdots P_a(q^{m_a})}{\lp 1 - q^{m_1} \rp^{k_1} \cdots \lp 1 - q^{m_a} \rp^{k_a}},
\end{align*}
where typically the $P_j$ are taken to be polynomials without constant term. The constants $k_j$ serve a role similar to weights for quasimodular forms; many analogies exist between the objects. The vector space of $q$-multiple zeta values is in fact an algebra, and it is known \cite{CraigIttersumOno} that these forms, when hit with natural actions of the symmetric group on weight-like components $k_j$, $\Q$-linearly generate the entire algebra of quasimodular forms. Already in this direction, work has been done to extend to higher levels \cite{Craig,Rose}, and there is much interest in congruences emerging from partition-theoretic questions coming from higher level variations of these $q$-multiple zeta values \cite{AmdeberhanAndrewsTauraso,AmdeberhanOnoSingh,SellersTauraso}.

It is therefore natural to ask whether the methods of this paper can be used to derive Sturm bounds for the spaces of mixed weight quasimodular forms. Following the Principle, this would follow from a Sturm bound for the space of mixed weight modular forms. As the author is not aware of any such result in the literature, we leave this as an open problem:

\begin{problem}
    Formulate a Sturm bound modulo primes for spaces of mixed weight modular forms.
\end{problem}

Now, motivated further by the work of Bachmann and K\"{u}hn \cite{BachmannKuhn}, we suggest that perhaps these techniques could extend even further in the universe of $q$-multiple zeta values. In particular, their work not only connects the $q$-multiple zeta values above to mixed-weight quasimodular forms; their work provides a graded algebra of $q$-multiple zeta values which are much larger than the algebra of of quasimodular forms. Since these objects have not been studied in great detail from the perspective of arithmetic $q$-series, but have been studied from the point of view of dimensionality conjectures \cite{BachmannKuhnDimension}, we propose the following course of study:

\begin{problem}
    Build an arithmetical theory or Sturm bound around the spaces of $q$-multiple zeta values.
\end{problem}

It would also be interesting to ask whether analogs exist for the Sturm bound in classical spaces of multiple zeta values, which have proper grading structures but which are not $q$-series.

\subsection{Comments on other domains for exploration}

A natural potential candidate for new Sturm bounds would be in the study of quasi-Jacobi forms \cite{IttersumOberdieckPixton}, which play a major role in the study of the elliptic genus of smooth manifolds. A stronger understanding of their Fourier coefficients would inform many aspects of the geometry embedded in these coefficients.

Another potential testing ground for extensions of this methods have recently emerged from the work of Bringmann, Pandey, and van Ittersum on their false and partial Eisenstein series \cite{BringmannPandeyIttersum}. These form natural graded algebras of $q$-series, generated by forms with (essentially) integer coefficients, they contain quasimodular forms as a subalgebra, and have $q$-coefficients that arise from formulas resembling some theories of quasimodular forms, and therefore are candidates for the kinds of arithmetic relations required for our method to apply. These $q$-series are also connected to unimodal sequences, which makes them natural candidates for the kinds of arithmetic which arise from modular-adjacent objects in this theory. We leave such questions for future work.

A larger and more difficult question would be whether such methods can be applied to infinite-dimensional spaces. Such an attempt would by necessity require much more complicated elements of operator theory in place of the classical linear algebra used in this work. Furthermore, one would need to be more clear about what is meant by a Sturm bound; in an infinite-dimensional context, the apparent answer is to ``improve on the $q$-expansion principle" so that one can determine a form $f$ uniquely in finite time based  on its $q$-expansion (and possibly other attached data). These are interesting questions which ought to be explored in these areas, and we will not speculate further.

\section*{Appendix: Code for $\Z$-basis}

The following code \cite{Code} can be used in SageMath to compute a set of quasimodular forms of weight $2k \geq 4$ that create a $\Z$-basis for $\widetilde{M}_{2k}^\Z$. Note that this algorithm will not in general produce anything ``triangular" in the usual sense. In particular, this uses Hermite normal form, where the output will be in a sense ``lower triangular" for us; instead of $q^i + O(q^{i+1})$, the basis will have form $a_0 + a_1 q^1 + \dots + a_i q^i + O(q^d)$ for $d$ the dimension in the underlying space. This code was written with the partial assistance of OpenAI’s ChatGPT (GPT-5 mini). The code performs the following operations:
\begin{enumerate}
    \item Enumerates all products of Eisenstein series $E_2^a E_4^b E_6^c$ contained in the space $\widetilde{M}_{2k}$;
    \item Computes Fourier expansions of each term up to $d = \dim \widetilde{M}_{2k}$ coefficients;
    \item Compile a matrix $M$ in which each column enumerates a truncated Fourier expansions;
    \item Compute the Hermite normal form of $M$, tracking at each stage all column operations performed. If the calculation fails, go back to (2) and include one extra term in each Fourier expansion;
    \item Output the particular quasimodular forms associated with the output columns one a valid Hermite normal form is achieved.
\end{enumerate}

\end{document}